\documentclass{amsart}

\usepackage{amsmath, amsthm, amsfonts,amssymb,amscd,mathrsfs}
\usepackage[centertags]{amsmath}
\usepackage{psfrag,graphicx}
\usepackage{multirow}
\usepackage[all]{xy}
\usepackage{tabularx}

\usepackage{remark}
\usepackage{indentfirst}

\usepackage{epic,eepic}
\usepackage{color}
\usepackage{ebezier}
\usepackage{enumerate}
\usepackage{mathrsfs}
\usepackage{graphicx}

\theoremstyle{plain}

\newtheorem{Thm}{Theorem}[section]
\newtheorem{Lem}[Thm]{Lemma}

\newtheorem{Prop}[Thm]{Proposition}
\newtheorem{Cor}[Thm]{Corollary}
\theoremstyle{plain}
\newtheorem{Def}[Thm]{Definition}
\newremark{example}[Thm]{Example}
\newremark{Rem}[Thm]{Remark}
\newremark{Construction}[Thm]{Construction}

\newtheorem{Exa}[Thm] {Example}

\numberwithin{equation}{section}

\newcommand{\Spec}{\mathrm{Spec}}
\newcommand{\cO}{\mathcal{O}}

\newcommand{\cL}{\mathcal{L}}

\newcommand{\cF}{\mathcal{F}}

\newcommand{\cM}{\mathcal{M}}
 
\newcommand{\sk}{\mathbf{k}}

\newcommand{\N}{\mathbb{N}}
\newcommand{\Z}{\mathbb{Z}}

\newcommand{\B}{\mathbb{B}}

\newcommand{\Bt}{\mathbb{B}^{\mathrm{tor}}}

\newcommand{\A}{\mathbb{A}_\sk^1}
\newcommand{\Aa}{(\mathbb{A}_\sk^1)^*}
\newcommand{\p}{\mathbb{P}_\sk^1}
\newcommand{\ea}{\mathbb{E}}
\newcommand{\ega}{\mathbb{E}^{\mathrm{gp}}}
\newcommand{\et}{\mathbb{E}^{\mathrm{tor}}}
\newcommand{\iso}{\mathbf{Iso}_U^{\mathrm{gp}}(J_U,E\times_\sk U)}
\newcommand{\is}{\mathbf{Iso}}

\usepackage{titletoc}

\newcommand{\red}{\color{black}}

\begin{document}
\bibliographystyle{plain}
\title{Surface fibrations with  large equivariant automorphism group}

\author{Yi Gu}

\email{sudaguyi2017@suda.edu.cn}
\address{School of Mathematical Sciences, Soochow University, Suzhou, 215006, P. R. of China}
\thanks{This work is supported by grant NSFC (No.~11801391) and NSF of Jiangsu Province (No.~BK20180832)}

\begin{abstract}
For a relatively minimal surface fibration $f: X\to C$, the equivariant automorphism group of $f$ is, roughly speaking, the group of automorphisms of $X$ preserving the fibration structure. We present a classification of such fibrations of fibre genus $g\ge 1$ with smooth generic fibre over an arbitrary algebraically closed field $\sk$ whose equivariant automorphism group is infinite.
\end{abstract}
\maketitle

\tableofcontents
\section{Introduction}
Let $\sk$ be an algebraically closed field and $f: X\to C$ be a relatively minimal surface fibration.  The equivariant automorphism group of $X/C$, denoted by $\ea(X/C)$ in this paper, is defined as the group
$$\ea(X/C):=\{\,\,(\tau,\sigma) \,\,|\,\, \tau\in \mathrm{Aut}_\sk(X), \sigma\in \mathrm{Aut}_\sk(C),\,\,\, f\circ \tau =\sigma\circ f\,\, \}$$
equipped with the natural composition law. See the following picture:
 $$
\xymatrix{ X\ar[rr]_\sim^\tau \ar[d]_f&& X \ar[d]^f\\
C \ar[rr]_\sim^\sigma && C.
}
$$By definition, we have a natural exact sequence of groups:
\begin{equation}\label{ext: intro-fundamental}
1\to \mathrm{Aut}_C(X)\to \ea(X/C)\stackrel{f_\sharp}{\to}\mathrm{Aut}_{\sk}(C),  \,\,\, f_\sharp: (\tau,\sigma)\mapsto \sigma, 
\end{equation}and let us denote by \begin{equation}
\B(X/C):=\mathrm{Im}(f_\sharp: \ea(X/C) \to \mathrm{Aut}_\sk(C))\subseteq \mathrm{Aut}_\sk(C).
\end{equation} 
Via another projection $$f^\sharp: \ea(X/C)\to \mathrm{Aut}_\sk(X), \,\,\, f^\sharp: (\tau,\sigma)\mapsto \tau,$$  the group $\ea(X/C)$ is naturally identified with a subgroup of $\mathrm{Aut}_\sk(X)$.  Along with its subgroup $\mathrm{Aut}_C(X)$ and quotient group $\B(X/C)$, the equivariant automorphism group $\ea(X/C)$  measures the symmetricity of this fibration and is an important invariant of $f$. Moreover, when $f: X\to C$ occurs naturally, {\it e.g.}, when $f$ is the Albanese fibration, the (pluri-)canonical fibration, then $\ea(X/C)$ is usually identified with $\mathrm{Aut}_\sk(X)$. 

In history, this group $\ea(X/C)$  has been studied by several people in different ways over the field of complex numbers. For example, Hu, Keum and Zhang provide dynamic criteria so that a subgroup $G\subseteq \mathrm{Aut}_\sk(X)$ can be embed to $\ea(X/C)$ for a suitable fibration $f: X\to C$ in \cite{Hu-Keum-Zhang15-criteria-for-equivariant-firbation}. When the fibration $f: X\to C$ comes naturally, the structures of $\ea(X/C)$ and its subgroup acting trivially on cohomology are studied by Cai in several papers   \cite{CaiJ-X06Automorphismofgenustwofibration,Cai-JX06,Cai-JX07,Cai-JX09Automorphisms-of-elliptic-surfaces}.  More recently, in papers of Prokhorov and Shramov \cite{P-S17, P-S19} and Shramov \cite{Shramov19-finite-group-action-on-elliptic-surfaces}, when studying the Jordan property of the automorphism group of algebraic surfaces, the authors are concerned with the finiteness of the group $\B(X/C)$ for several different kinds of fibrations $f: X\to C$.

The aim of this paper is to give a complete classification of those relatively minimal surface fibration $f: X\to C$ of positive fibre genus so that $\ea(X/C)$ is infinite, in all characteristic. To be more precisely, in positive characteristic there are fibrations  with singular general fibres, {\it e.g.}, the quasi-elliptic fibrations.  Fibrations with singular general fibres are wild and need  a different treatment beyond this paper. By excluding these fibrations, we actually classify  fibrations with smooth general fibres whose equivariant automorphism group is infinite. For this purpose, let us identify ``{\bf surface fibration}" with ``{\bf surface fibration having smooth general fibre}" from now on for simplicity. 

As can be easily expected, surface fibration with infinite equivariant automorphism group are isotrivial. For isotrivial surface fibrations, we mention the following definition.
\begin{Def}[cf. \protect{\cite[\S~1]{Serrano96}}]\label{Def: standard fibration}
(1). An isotrivial surface fibration $f: X\to C$ with smooth general fibre $F$ is called   \emph{pre-standard} if there is a suitable open subset $V\subseteq C$ and a finite group scheme $G$ acting  faithfully both on $F$ and another irreducible \emph{quasi-projective} curve $V'$  so that we have  the  left commutative diagram above. Here $G$ acts on $F\times_\sk V'$ diagonally.
$$
\xymatrix{
  X\ar[d]^f &\ar@{_{(}->}[l] X_V \ar[d]^{f_V} \ar[rr]^{\simeq} && (F\times_\sk V')/G\ar[d]^{p_2}  & X \ar@{-->}[rr]^{\text{birational} \,\,\,\,\,\,\,} \ar[d]^f &&(F\times C')/G\ar[d]^{p_2}\\
     C&\ar@{_{(}->}[l] V \ar[rr]^\simeq && V'/G  & C\ar[rr]^\simeq && C'/G                                       .
}
$$
(2).  An pre-standard surface fibration $f: X\to C$ is called  \emph{standard} if we can take $V=C$ and $V'$ is a smooth. See the left commutative diagram. 

In above settings,  we call the pair $(F,C',G)$ ({\it resp.} $(F,V,V',G)$) as a ({\it resp.} pre-)standard realisation of $f: X\to C$.
\end{Def}
A standard fibration is automatically pre-standard and the converse is also true in characteristic zero  or if  $G$ is \'etale. A useful result in characteristic zero, given in \cite{Serrano96}, is that an isotrivial fibration is always standard (with a key step given in \cite[\S~VI]{Beauville-book} for the genus one case). This result helps substantially in some classification problem ({\it e.g.}, the classification of the so-called {\it hyperellipitc} surface, \cite[pp~148]{BHPV04}) and in calculations of invariants of $f: X\to C$, such as $\mathrm{Aut}_\sk(X), K_X^2, \chi(\cO_X), \cdots$ (see, {\it e.g.}, \cite{Bennett-Miranda90, Polizzi15}).

In positive characteristic, when the fibre genus of an isotrivial fibration $f: X\to C$ is at least two, the automorphism group of the general fibre $F$ is rigid and thus such a fibration remains standard. But when it comes to isotrivial fibration  of genus one, it turns out to be different.
\begin{Prop}[Proposition~\ref{Prop: isotrivial=pre standard}]\label{Prop: main in front}
A relatively minimal isotrivial genus one surface fibration is pre-standard but not necessarily standard.
\end{Prop}
Example of non-standard genus one fibration is given in Example~\ref{Exa: concrete}. This proposition is somehow the weaker analogue of \cite[Lem.~VI.10(2)]{Beauville-book} and \cite{Serrano96} in positive characteristic. It should also be noted that the geometry ({\it e.g.}, their automorphism group and numerical invariants such as the Kodaira dimension of $X$) of a general pre-standard genus one fibration is difficult to handle. With some characteristic-$p$ pathologies behind (cf. Example~\ref{Exa: concrete} for the Kodaira dimension), the classical methods working with standard fibrations in characteristic zero (cf. \cite{Polizzi15}) does not apply in general. 

%So the main difficulty of our problem to classify relative minimal surface fibration with infinite equivariant automorphism group lies in the genus one case. In this case, without a structure theorem strong enough at hand, we first have to build one by ourselves based on the infiniteness assumption of the equivariant automorphism group rather than the isotriviality alone and then can we run the classification procedure after the structure theorem built. 
 
Now let us return to the classification problem aforementioned. By (\ref{ext: intro-fundamental}), the equivariant automorphism group $\ea(X/C)$ is infinite if and only if either $\mathrm{Aut}_C(X)$  or $\B(X/C)$ is infinite. A first result of the paper is an infinite criterion for the group $\mathrm{Aut}_C(X)$.
\begin{Thm}[Theorem.~\ref{Thm: main1}]\label{Thm: main1 in front}
Let $f:X\to C$ be a relatively minimal surface fibration of genus $g\ge 1$ over $\sk$. Then $\mathrm{Aut}_C(X)$ is infinite ({\it resp.} infinitely generated) if and only if $g=1$ and the Jacobian fibration (cf. Definition~\ref{Def: Jacobian fibration}) of $f: X\to C$ has  infinitely many sections ({\it resp}. is trivial).
\end{Thm} 
%After this theorem, we discuss the classification of relatively minimal genus one fibrations with trivial Jacobian in \S~\ref{Sec:Fibration with large relative automorphism group}. If the fibration is at the same time of standard type, we present a  complete classification  in Theorem~\ref{Thm: standard genus one fibration with trivial Jacobian}. In positive characteristics, we  present examples of  genus one fibrations with trivial Jacobian in \S~\ref{Subsubsec: Structure on genus one fibration in positive characteristics} that fail to be standard type. In general, pre-standard fibrations of genus one is too complicated there should be no reasonable `complete classification' to them. Instead, we provide a coarse characterization of such fibration in \S~\ref{Subsubsec: Structure on genus one fibration in positive characteristics}.

Next we turn to the most important part of the paper: the classification of relatively minimal surface fibration $f:X\to C$ of positive fibre genus $g$ admitting an infinite $\B(X/C)$.  The main result is the following structure result:
\begin{Thm}[Theorem~\ref{Thm: main in g>2} \& \ref{Thm: main+main}]\label{Thm: main in front}
Let $f: X\to C$ be a relatively minimal surface fibration of genus $g\ge 1$ with infinite $\B(X/C)$. Taking $U\subseteq C$ to be the smooth locus of $f$, then there are following data:
\begin{enumerate}[(a).]
\item  a smooth curve $F$ of genus $g$;

\item   a finite subgroup scheme $G'\subseteq \mathbf{Aut}_\sk(F)$ over $\sk$, a $G'$-torsor $\nu': T'\to U$  such that

\begin{enumerate}[(i)]
\item $T'$ is irreducible and smooth over $\sk$;

\item the torsor $T'$ over $U$ admits an infinite equivariant torsor-automorphism group $\et(T'/U)$ (cf. Definition~\ref{Def: equivariant automorphism}); 

\item the $G'$-action on $T'$ extends to its normal compactification $T'\subseteq C'$.
\end{enumerate}
\end{enumerate}
With these data the fibration $f_U$ and $f$ is given as follows:
$$
\xymatrix{
X_U\ar@{=}[rr] \ar[d]^{f_U} && (F\times_\sk T')/G' \ar[d]^{p_2}\\
U  \ar@{=}[rr]                         &&  T'/G'
}\,\,\,\xymatrix{
X \ar@{-->}[rrrr]^{\text{minimal resolution}}    \ar[d]^{f } &&&& (F\times_\sk C')/G' \ar[d]^{p_2}\\
 C  \ar@{=}[rrrr]                         &&&&  C'/G'
}
$$
Here, with the $G'$-actions on $F,T',C'$ mentioned above, the group scheme $G'$ acts on $F\times_\sk T'$ and $F\times_\sk C'$ both diagonally.
\end{Thm}  
As a result, such fibrations are standard and they admits a very restrictive standard realization $(F,C',G)$. The pairs $(G',\nu: T'\to U)$ above is classified in Proposition~\ref{Prop: classification of torsors with infinite Bt} and as a result, we manage to present a complete classification of relatively minimal surface fibration $f: X\to C$ of genus $g\ge 1$ with infinite $\B(X/C)$ in Theorem~\ref{Thm: main in g>2} ($g\ge 2$) in \S~\ref{Sec: B infinite, g>1} and Theorem~\ref{Thm: classification of genus one} ($g=1$) in \S~\ref{Subsec: classification in genus one}.  The two classification theorems are rather long, we shall not include it here.

A remarkable consequence of Theorem~\ref{Thm: main in front} is that for such a fibration $f:X\to C$,   the Kodaira dimension of $X$ is restrictive.
 \begin{Cor}[Corollary~\ref{Cor: last}]\label{Cor: 2020}
Under the same settings of Theorem~\ref{Thm: main in front} and assume $g=1$, then we have $\kappa(X)=0$ if $g(C)=1$ and $\kappa(X)=-\infty$ if $g(C)=0$. In particular, $\kappa(X)\le 0$. 
\end{Cor}
Recall when $S$ is a minimal surface of Kodaira dimension one, its Iitaka fibration  $f: S\to C$ is either a relatively minimal genus one (with smooth general fibres) or quasi-elliptic (with singular general fibres) fibration. This fibration is $\mathrm{Aut}_\sk(S)$-equivariant so that $\ea(S/C)=\mathrm{Aut}_\sk(S)$. Corollary~\ref{Cor: 2020} then says that the image of the homomorphsim $$f_\sharp: \mathrm{Aut}_\sk(S)=\ea(S/C)\to \mathrm{Aut}_\sk(C)$$ is finite if $f$ is a genus one fibration.  In other words, this corollary gives a partial positive characteristic generalization of  \cite[Proposition~1.2]{P-S19}, which is considered as `{\it one of the main steps in the proof of Theorem~1.1 (the main theorem of \cite{P-S19})}'. And it then gives the Jordan property of $\mathrm{Aut}_\sk(S)$.
\begin{Cor}[Corollary~\ref{Cor: Jordan}]
Let $S$ be a relatively minimal surface of Kodaira dimension one, suppose its Iitaka fibration  $f: S\to C$ has smooth general fibre, then $\mathrm{Aut}_\sk(S)$ has Jordan property. Namely there is a uniform bound $N(S)\in \N_+$ such that any finite subgroup of $\mathrm{Aut}(S)$ contains an Abelian subgroup of index at most $N(S)$.  
\end{Cor}

\newpage

 \begin{center}
 {\bf Conventions}
 \end{center}
 \begin{enumerate}[a)]
 \item $\sk$ is an algebraically closed field; 

 \item a surface fibration $f:X\to C$ over $\sk$ is a proper flat morphism from a smooth projective surface $X$ onto a curve $C$ with smooth connected geometric generic fibre. It is called relatively minimal if  there is no vertical $(-1)$-curves;

 \item for either an Abelian group or Abelian group scheme $G$, we denote by $$G[n]:=\mathrm{Ker} (G\stackrel{\cdot n}{\to} G);$$
 
 \item when $\mathrm{char}.(\sk)=p>0$, we denote by $F_{Y/\sk}: Y\to Y^{(p)}:=Y\times_{\sk, F_\sk} \sk$ the relative Frobenius map of any scheme $Y$ over $\sk$;

 \item denote by $(\mathbb{A}_\sk^1)^*:=\mathbb{A}_\sk^1\backslash\{0\}$;

 \item for a curve $F$, $\mathrm{Aut}_\sk(F)$ is the automorphism group of $F$ while  $\mathbf{Aut}_\sk(F)$ is the group scheme representing the automorphism functor of $F$ (cf. \S~\ref{Sub: auto of curves}); 
  
 \item symbol with underline means the associated functor, {\it i.e.}, $\underline{X}$ is the representable functor of $X$.

 \end{enumerate}

\section{Preliminaries}\label{Sec: Preliminaries}
\subsection{Torsors and  quotient of Schemes}
\subsubsection{Torosrs}
We first recall the notion of  (fppf-)torsors.  
 
\begin{Def}[Torsors, \protect{\cite[\S~6.4]{BLR90-Neron-Models}}]
Let $Y$ be a variety over $\sk$, $G$ be a flat group scheme of finite type over $Y$ and $T$ be a flat $Y$-scheme of finite type with a (left-)$G$-action $\mu:G\times_Y T\to T$. Then $T$ is called  a (fppf-)torosr under $G$ over $Y$ if the map $$G\times_YT\stackrel{(\mu, p_2)}{\longrightarrow} T\times_YT,\,\,\, (g,t)\mapsto (g\cdot t, t)$$ is an isomorphism.  

By abuse of language, when $G$ is a group scheme over $\sk$,  we also call a $G\times_\sk Y$-torsor over $Y$ as a $G$-torsor over $Y$ for simplicity. 
\end{Def} 

\medskip 

\begin{Rem}\label{Rem: on torsors}
If $G$ is an \'etale group scheme over $\sk$, then $G$ is the same as an abstract group because $\sk$ is algebraically closed. In this case,  a $G$-torsor over $Y$ is nothing but an \'etale Galois cover $\nu: T\to Y$ with Galois group $G$.
\end{Rem}

Then we recall the notion of extension and reduction of the structure group of a torsor.
 
\begin{Def}[Extension and Reduction of the structure group]\label{Def: reduction of structure group}
Let $G'\subseteq G$ be a finite flat subgroup scheme over $Y$. Then given any  $G'$-torsor $T'$, there defines a natural free $G'$-action on $G\times_Y T'$ functorially described as follows:
$$g'\cdot (g,t')= (gg'^{-1},g'(t')), \,\,\,\forall\,\,\, X\in \mathrm{Sch}/Y, g\in G(X),g\in G'(X), t'\in T'(X).$$ 
Then the $G$-action on $G\times_Y T'$ given by $$g\cdot (h,t')=(gh,t'),\,\,\,\forall\,\,\, X\in \mathrm{Sch}/Y, g,h\in G(X), t'\in T'(X)$$  is commutative to the previous diagonal $G'$-action. Thus there defines a natural $G$-action on the quotient scheme $T:=(G\times_Y T')/G'$ (cf. Theorem~\ref{Thm: quotient by gp} below) making $T$ a $G$-torsor (in fact, $T$ represents the functor $\underline{G}\times^{\underline{G'}} \underline{T'}$).  The $G$-torsor $T$ is called the extension of structure group by $G'\to G$ of $T'$ and conversely, $T'$ is called the reduction of structure group by $G'\to G$ of $T$.  For simplicity, we simply call that $T$ reduces to $T'$ and $T'$ extends to $T$. 
\end{Def}

\begin{Exa}\label{Exa: typical example of extension and reduction of torsors}
Let $Y$ be a smooth variety over $\sk$ and $\nu: T\to Y$ be an \'etale Galois cover with Galois group $G$ and $T'\subseteq T$ be an irreducible component. Then $T$ is a $G$-torsor over $Y$, and since $\nu|_{T'}: T'\to Y$ is again an \'etale Galois cover with Galois group $G'=\{g\in G| g(T')=T'\}$, $T'$ is a $G'$-torsor over $Y$. By construction, it is easy to figure out that $T$ reduces to $T'$ or equivalently $T'$ extends to $T$ by the group extension $G'\subseteq G$.
\end{Exa}

\subsubsection{Quotient a scheme by a finite group scheme}\label{Subsubsec: quotient by gp} Let $Y$ be a quasi-projective $\sk$-variety and $G$ be a finite flat group scheme over $Y$ acting on a quasi-projective $Y$-scheme $X$. Then there defines a functor
$$\underline{X}/\underline{G}: \mathrm{Sch}/Y\to  {\mathrm{\bf Set}}, \,\,\, T\mapsto X(T)/G(T), \,\,\,\forall\,\,\, T\in \mathrm{Sch}/Y.$$
\begin{Thm}[\protect{\cite[\S~V.10, Thm.~10.3.1]{SGA3}}]\label{Thm: quotient by gp}
(1). The category quotient   $\underline{X}/\underline{G}$ is represented by a scheme $Q$, denoted by $X/G$. 

(2). When the $G$-action on $X$ is free,   $\underline{Q}$ is the (fppf-)sheafification of $\underline{X}/\underline{G}$, the quotient map $\pi: X\to Q$ is a (fppf-)torsor under $G$ and the formation of quotient schemes commutes with any base change over $Q$.  
\end{Thm}
As a consequence, we have the following corollary.
\begin{Cor}\label{Cor: quotient by gp}
(1). Suppose $G$ is a finite group scheme over $\sk$ acting freely on a smooth quasi-projective $\sk$-variety $X$, then the quotient variety $X/G$ is also smooth over $\sk$.

(2). Suppose $G$ is a finite group scheme over $\sk$ acting on two smooth quasi-projective $\sk$-varieties $X,Y$ and hence acts diagonally on $X\times_\sk Y$. Then if the $G$-action on $Y$ is free, we have the following commutative is Cartesian:
$$
\xymatrix{
X\times_\sk Y \ar[rr] \ar[d]^{p_2}&& (X\times_\sk Y)/G \ar[d]^{p_2}\\
Y \ar[rr] && Y/G 
}
$$
In particular, the  morphism $(X\times_\sk Y)/G\to Y/G$ is smooth and all the  closed fibres of this morphism are isomorphic to $X$.
\end{Cor}

\subsubsection{Quotient a scheme by a foliation}\label{Subsubsec: foliation}Next, we recall the quotient of a scheme by a foliation.
\begin{Def}[\protect{\cite[\S~I]{Ekedahl88-Canonical-models}}]
Let $Y$ be a smooth variety over $\sk$ of positve characteristic. A foliation $\cF\subseteq \mathcal{T}_{Y/\sk}$  on $Y$ is a saturated  sub-$p$-Lie  algebra. A foliation $\cF$ is smooth if it is at the same time a subbundle of  $\mathcal{T}_{Y/\sk}$, {\it i.e.}, $\mathcal{T}_{Y/\sk}/\cF$ is locally free.
\end{Def}

For a foliation $\cF$, the quotient scheme $X:=Y/\cF$ is constructed as a ringed spaced $(X,\cO_X)$ as below: 
\begin{itemize}
\item the underlying space $X=Y$;

\item $\cO_X\subseteq \cO_Y$ is the subalgebra consists of functions killed by any local derivatives in $\cF$.
\end{itemize}

\begin{Thm}[\protect{\cite[\S~3]{Ekedahl85-Foliations}}]\label{Thm: Ekedahl}
If $\cF$ is a smooth foliation on $Y$, then $X=Y/\cF$ is also smooth over $\sk$. For the quotient map $\pi: Y\to X$, we have:
\begin{equation}\label{formula: canonical bundle and foliation}
\pi^*K_X=K_Y-(p-1)c_1(\cF).
\end{equation}
\end{Thm}

\subsubsection{Examples of group actions}
Let $t$ be the affine coordinate of $\A\subseteq \p$, then we have two natural group actions as follows.  

\begin{Exa}\label{Exa: example of group actions}
(1)[The additive action]. The group scheme $\mathbb{G}_a$ can acts on $\p$ naturally via 
$$(\lambda, t)\mapsto t+\lambda.$$
Here $\lambda$ is the coordinate of $\mathbb{G}_a$.  

(2)[The multiplicative action]. The group  scheme $\mathbb{G}_m$ can acts on $\p$ naturally via  $$(\lambda, t)\mapsto \lambda\cdot t.$$ Here $\lambda$ is the  coordinate of $\mathbb{G}_m$.
\end{Exa}
We will frequently refer to these two actions in this paper. Finally, we recall a characterization of $\alpha_p$ or $\mu_p$ action on varieties.
\begin{Prop}[see \protect{\cite[\S~3]{Tziolas17-Quotient-of-schemes}}]\label{Prop: action by alphap}
Let $Y$ be a smooth variety over $\sk$, then an $\alpha_p$ ({\it resp.} $\mu_p$) action on $Y$ is the same as a regular derivative $\delta\in H^0(Y,\mathcal{T}_{Y/\sk})$ so that $\delta^p=0$ ({\it resp.} $\delta^p=\delta$). 
\end{Prop}
Under this correspondence, the $\alpha_p$ ({\it resp.} $\mu_p$)-action on $Y$ is free if and only if the associated derivative $\delta$ is nowhere vanishing. When this is the case, the derivative $\delta$ gives to a smooth foliation $\cF(\delta):=\cO_Y\cdot \delta\subseteq \mathcal{T}_{Y/\sk}$  called as the foliation associated to $\delta$.  Then the quotient by foliation $Y/\cF(\delta)$ is the same as the quotient by the group scheme action $Y/\alpha_p$ ({\it resp.}  $Y/\mu_p$).

\subsection{Equivariant automorphisms}
We introduce the notation of equivaraint automorphism group of a morphism which is compatible with the one given in the introduction.
\begin{Def}\label{Def: equivariant automorphism}
(1). Given an arbitrary morphism $h:X\to Y$ over $\sk$, the equivariant automorphism group of $h:X\to Y$, denoted by $\ea (X/Y)$, is the set of pairs $$\{(\tau,\sigma)|\tau\in \mathrm{Aut}_\sk(X), \sigma\in \mathrm{Aut}_\sk(Y) \,\, \text{with} \,\, h\cdot \tau=\sigma\cdot h \}$$
equipped with its natural group structure.

(2). If $h: G\to Y$ is a group scheme, then the equivariant group automorphism group of $h:G\to Y$, denote by $\ega (G/Y)$, is the subgroup of $\ea (G/Y)$ consisting of $(\tau,\sigma)$ such that $(\tau,h): G\to G\times_{Y,\sigma} Y$ is an isomorphism of group schemes.
$$
\xymatrix{
G\ar[rrd]^{(\tau,h)} \ar@/^2em/[rrrrd]^h \ar[rrdd]_\tau \\
&& G\times_{Y,\sigma} Y \ar[rr]\ar[d]  && Y\ar[d]^\sigma \\
&& G\ar[rr]^h && Y
}\,\,\,\,\,\
\xymatrix{
G\times_\sk T\ar[d]^{\mu} \ar[rr]^{\mathrm{id}\times \tau} &&G\times_\sk T  \ar[d]^{ \mu}\\
T\ar[rr]^{\tau} \ar[d]^h&&  T \ar[d]^h \\
 Y \ar[rr]^\sigma && Y}
$$

(3). If $G$ is a group scheme over $\sk$, $Y$ is a $\sk$-variety and $h: T\to Y$ is a  $G$-torsor. Then the equivariant torsor automorphism group of $h:T\to Y$, denote by $\et(T/Y)$, is the subgroup of $\ea (T/Y)$ consisting of $(\tau,\sigma)$ such that $(\tau,h):  T\to T\times_{Y,\sigma} Y$ is an isomorphism of $G$-torsors. Equivalently, we have the right commutative diagram above for the pair $(\tau,\sigma)$,  where $\mu: G\times_\sk T\to T$ is the group action morphism.
\end{Def}
These notions will be substantially used in the following of the paper. To be familiar with these notions, we check an easy example.
\begin{Exa}\label{Exa: ea and centalizer}
Let $\pi: X\to Y$ be an \'etale Galois cover of smooth $\sk$-varieties  with Galois group $G=\mathrm{Aut}_Y(X)\subseteq \mathrm{Aut}_\sk(X)$. Then $X$ is considered as a $G$-torsor over $Y$ (cf. Remark~\ref{Rem: on torsors}(1)). In this case, we have identifications of $\et(X/Y)$ and $\ea(X/Y)$ shown in the following picture: 
$$
\xymatrix{
\et(X/Y) \ar[r]^{p_1}_\simeq \ar@{^{(}->}[d]& C_G(\mathrm{Aut}_\sk(X)) \ar@{^{(}->}[d]\ar@{^{(}->}[r] & \mathrm{Aut}_\sk(X)\ar@{=}[d] &  \ea(X/Y)\ar[r]^{p_1} &\mathrm{Aut}_\sk(X) \\ 
\ea(X/Y) \ar[r]^{p_1}_\simeq  & N_G(\mathrm{Aut}_\sk(X)) \ar@{^{(}->}[r] & \mathrm{Aut}_\sk(X) & (\tau,\sigma) \ar@{|->}[r]^{p_1} & \tau
}
$$ 
\end{Exa}

\subsection{Isomorphism and Automorphism of curves}\label{Sub: auto of curves}
We first declare a convention. Let $F$ be a smooth projective curve over $\sk$. Then 
\begin{itemize}
\item $\mathrm{Aut}_\sk(F)$ is to denote the $\sk$-automorphism group of $F$.

\item  $\mathbf{Aut}_\sk(F)$ is to denote the group scheme representing the functor (cf. \cite[\S~5.6.2]{FGA-explained})
$$\underline{\mathrm{Aut}}_\sk(F): \mathrm{Sch}/\sk \to \mathbf{Group}, \,\,\, T\mapsto \mathrm{Aut}_T(F\times_\sk T).$$
\end{itemize}
Namely, we have $$\mathbf{Aut}_\sk(F)(\sk)=\mathrm{Aut}_\sk(F).$$ When $g\ge 2$, it is well known that the group scheme $\mathbf{Aut}_\sk(F)$ is finite and \'etale over $\sk$. So it makes no difference to use $\mathbf{Aut}_\sk(F)$ and $\mathrm{Aut}_\sk(F)$ interchangeably. 

\subsubsection{Automorphism of elliptic curves}
Let $E$ be an elliptic curve (with group scheme structure fixed). Then the functor 
$$\underline{\mathrm{Aut}}^{\mathrm{gp}}_\sk(E): \mathrm{Sch}/\sk \to \mathbf{Group}, \,\,\, T\mapsto \mathrm{Aut}^{\mathrm{gp}}_T(E\times_\sk T)$$ is represented by a finite \'etale subgroup scheme $\mathbf{Aut}_\sk^{\mathrm{gp}}(E)\subseteq \mathbf{Aut}_\sk(E)$. As this subgroup scheme is finite \'etale, it makes no different to use  $\mathbf{Aut}_\sk^{\mathrm{gp}}(E)$ and $\mathrm{Aut}_\sk^{\mathrm{gp}}(E)$ interchangeably. Then it is well known $\mathbf{Aut}_\sk(E)$ admits a natural  semi-direct product:
 \begin{equation}\label{Equ: semi-direct decomposition of E}
 \mathbf{Aut}_\sk(E)=E\rtimes \mathrm{Aut}_\sk^{\mathrm{gp}}(E),
 \end{equation}
where $E\subseteq \mathbf{Aut}_\sk(E)$ is the translation subgroup scheme.  It is well known that 
 \begin{Prop}\label{Prop: classification of subgroup acting freely on E}
 Any finite subgroup scheme of $\mathbf{Aut}_\sk(E)$ acting freely on $E$ is contained in the translation subgroup scheme $E\subseteq \mathbf{Aut}_\sk(E)$.
 \end{Prop}
\subsubsection{Automorphism of affine curves}
Let $Y$ be  either $\A$ or $\Aa$ now. As $Y$ is no longer projective, the automorphism functor $$\underline{\mathrm{Aut}}_\sk(Y):\mathrm{Sch}/\sk \to \mathbf{Group}, \,\,\, W\mapsto \mathrm{Aut}_W(Y\times_\sk W)$$ is no longer representable. In this case, we consider the sub-functor of the automorphism functor consisting of linear automorphisms. Namely, we use the canonical embedding $Y\hookrightarrow \p$ and take $\Delta=(\p\backslash Y)_{\mathrm{red}}$. Then we define a subfunctor $\underline{\mathrm{LAut}}_\sk(Y)\subseteq \underline{\mathrm{Aut}}_\sk(Y)$
$$\underline{\mathrm{LAut}}_\sk(Y)=\underline{\mathrm{Aut}}_\sk(\p, \Delta): \mathrm{Sch}/\sk \to \mathbf{Group},$$
$$ W\mapsto   \{\,\, \sigma \in \mathrm{Aut}_W(\p \times_\sk W)\,\,|\,\, \sigma|_{\Delta\times_\sk W}=\mathrm{id} \,\,  \} $$
that is represented by a subgroup  scheme $\mathbf{LAut}_\sk(Y)\subseteq \mathbf{PGL_{2,\sk}}$. We call this subgroup scheme $\mathbf{LAut}_\sk(Y)$ the linear automorphism group scheme.  The explicit description of the linear automorphism group scheme is as follows: 
\begin{itemize}
\item when $Y=\Aa$, $\mathbf{LAut}_\sk(Y)=\mathbb{G}_m$ which acts on $\Aa\subseteq \p$ multiplicatively (cf. Example~\ref{Exa: example of group actions});

\item when $Y=\A$, then $\mathbf{LAut}_\sk(Y)=\mathbb{G}_a\rtimes \mathbb{G}_m$, where $\mathbb{G}_a$ and $\mathbb{G}_m$ acts on $\A\subseteq \p$ additively and multiplicatively (cf. Example~\ref{Exa: example of group actions}) respectively.
\end{itemize}

 Now let $R$ be a finite $\sk$-algebra. We shall study the $R$-automorphism group of $Y_R:=Y\times_\sk R$.

\begin{Prop}\label{Prop: infinitesimal automorphism}
Let $H\subseteq \mathrm{Aut}_\sk(Y)\subseteq \mathrm{Aut}_R(Y_R)$ be an arbitrary infinite subgroup, $\varphi: Y_R\to Y_R$ be an $R$-automorphism contained in the centralizer group of $H$  in $\mathrm{Aut}_R(Y_R)$. Then  $\varphi$ is contained in $\mathbf{LAut}_\sk(Y)(R)$.
\end{Prop}
\begin{proof}
In the following, in either $Y=\A$ or $\Aa$ case, we denote by $t$ its natural coordinate function.

{\bf Case: $Y=\Aa$.}

In this case we have an exact sequence: 
\begin{equation}\label{Ext: aut Gm}
0\to \sk^* \to \mathrm{Aut}_\sk(Y) \to \mathbb{Z}/2\Z\to 0.
\end{equation}Here $\sk^*$-acts on $Y$ multiplicatively (cf. Example~\ref{Exa: example of group actions}). Then by assumption, the group  $H\cap \sk^*$ is also infinite. We write $\varphi(t)=\sum\limits_{i=-n}^n \mu_it^i, \mu_i\in R$ for some $n\in \N_+$. Then for any $\lambda\in H\cap \sk^*$, by the centralizer assumption we have 
$\varphi(\lambda t)=\lambda \varphi(t)$. In all, the equation $(\lambda^{i-1}-1)\mu_i=0$ holds for infinitely many $\lambda\in \sk$ for any $i$. Hence we must have $\mu_i=0$ for all $i$ except $i=1$. We are done.

{\bf Case: $Y=\A$.} 

We have $\mathrm{Aut}_\sk(\A)=\{\phi_{a,b}: t\mapsto at+b | a\in \sk^*,b\in \sk\}$. 
We  can write $\varphi(t)=\sum\limits_{i=0}^n \mu_it^i, \mu_i\in R$ for some $n=\deg(\varphi)\in \mathbb{N}_+$ (here $\deg(\varphi)$ is the degree of the leading non-zero coefficient rather than the invertible coefficient). By the centralizer assumption, we have
$$\varphi(at+b)=a\varphi(t)+b,\,\,\, \forall \,\,\,\phi_{a,b}\in H.$$ 

Assuming $n=\deg(\varphi)\ge 2$, we need to obtain a contradiction. Firstly by comparing the leading coefficient in the above equation we have $a^n=a \in \sk$ for any $\phi_{a,b}\in H$. Therefore, there are only finitely many choice of $a$ and following the infiniteness of $H$, the intersection of $H$ with the subgroup $\sk:=\{\phi_{1,b}| b\in \sk\}$ is infinite. Namely $\varphi(t+b)=\varphi(t)+b$ holds for infinitely many $b\in \sk$.  Taking a $\sk$ basis $1=e_1,\cdots, e_m$ of $R$ such that $e_2,\cdots, e_m$ form a basis of a maximal ideal $\mathfrak{m}$ of $R$. Then we can find $\varphi_i(t)\in \sk[t]$ so that
$$\varphi(t)=\varphi_1(t)e_1+\varphi_1(t)e_2+\cdots+\varphi_m(t)e_m, \,\,\, \varphi_i(t)\in \sk[t].$$
Now the equation $\varphi(t+b)=\varphi(t)+b$ is equivalent to $\varphi_i(t+b)=\varphi_i(t)+\delta_{1,i}b$. First note that $\varphi_1: \A\to \A$ is the induced isomorphism $$\varphi_1: \mathbb{A}_R^1\otimes_R \sk \stackrel{\varphi}{\to}\mathbb{A}_R^1\otimes_R \sk, \,\,\, \sk=R/\mathfrak{m}$$ for the maximal ideal $\mathfrak{m}$. so $\deg(\varphi_1)=1$. Then for $i\ge 2$, the equation $\varphi_i(t+b)=\phi_i(t)$ for infinitely many $b\in \sk$ forces $\varphi_i(t)$ to be a  constant polynomial. Therefore  $n=\mathrm{max}\{deg(\varphi_i(t))| i=1,2,\cdots,m\}$ is at most one, a contradiction.
\end{proof}

\begin{Cor}\label{Cor: extension of infinitesimal automorphism}
Let $G$ be a finite group scheme over $\sk$ admitting a faithful action on $Y=\A$ or $\Aa$. Suppose the group action of $G$ is  commutative with an infinite subgroup $H\subseteq \mathrm{Aut}_\sk(Y)$.  Then $G$ is a subgroup scheme of  $\mathbf{LAut}_\sk(Y)$. In particular, the $G$-action on $Y$ extends to the compactification $Y\subseteq \p$.
\end{Cor}
\begin{proof} 
The group action $\mu: G\times_\sk Y\to Y$ gives an $R:=\cO_{G}$-automorphism of $Y$:
$$
\xymatrix{
G\times_\sk Y \ar[d]^{p_1}\ar[rrrr]^{\Phi:=\mathrm{id}\times \mu} &&&& G\times_\sk  Y \ar[d]^{p_1} && (g,y)\mapsto (g, g\cdot y) \\
G\ar@{=}[rrrr] &&&& G}
$$
This automorphism by definition commutes with $H$. By Proposition~\ref{Prop: infinitesimal automorphism}, $G$ is a subgroup scheme of $\mathbf{LAut}_\sk(Y)$.
\end{proof}

Finally, let us figure out the finite subgroup schemes of  $\mathbf{LAut}_\sk(Y)$ that acts freely on $Y$.
\begin{Prop}\label{Prop: classification of subgroup acting freely on U}
(1). When $Y=\Aa$, the  group scheme $\mathbf{LAut}_\sk(Y)=\mathbb{G}_m$ acts freely on $Y$. In particular, any finite subgroup scheme of $\mathbb{G}_m$ acts on $U$ freely.

(2). When $Y=\A$, with respective to the semi-direct product $\mathbf{LAut}_\sk(Y)=\mathbb{G}_a\rtimes \mathbb{G}_m$, a subgroup scheme of $\mathbf{LAut}_\sk(Y)$ acts freely on $Y$ if and only if it is contained in $\mathbb{G}_a$.
\end{Prop}
The proof is elementary and we left it to the readers.

\section{Torsors with large equivariant automorphisms}\label{Section: torsors}
\subsection{Statement of the main theorem}
The aim of this section is to give a classification of the pair $(G, \nu:T\to U)$, where 
\begin{enumerate}[(i)]
\item $G$ is a finite group scheme over $\sk$ whose identity component $G^0$ is Abelian;

\item $U$ is an irreducible, smooth curve over $\sk$;

\item $\nu: T\to U$ is a  $G$-torsor with infinite $\et(T/U)$ (cf. Definition~\ref{Def: equivariant automorphism}).
\end{enumerate}
Note that there is an exact sequence  (cf. \cite[Lem.~4.1]{Brion11})
$$1\to \mathrm{Hom}_\sk^G(T,G)\to \et(T/U)\stackrel{\nu_\sharp}{\to} \mathrm{Aut}_\sk(U), \,\,\, \nu_\sharp: (\tau,\sigma)\mapsto \sigma,$$
here the group $\mathrm{Hom}_\sk^G(T,G)\subseteq \mathrm{Hom}_\sk(T,G)$ is the subgroup consisting of equivariant $G$-homomorphism from $T$ to $G$ ($G$ acts on itself by conjugation). 

\begin{Lem} 
The subgroup $\mathrm{Hom}_\sk^G(T,G)$ is finite under our assumption. 
\end{Lem}
\begin{proof}
 As $G$ is finite, we have a natural semi-direct decomposition: $G=G^0\rtimes G(\sk)$ and thus obtain  another natural exact sequence:
$$
1\to\mathrm{Hom}_\sk^{G}(T,G^0)\to \mathrm{Hom}_\sk^G(T,G) \to \mathrm{Hom}_\sk^G(T,G(\sk)).
$$
The quotient target group $\mathrm{Hom}_\sk^G(T,G(\sk))\subseteq \mathrm{Hom}_\sk(T,G(\sk))$ is automatically finite, and it suffices to prove that  $\mathrm{Hom}_\sk^{G}(T,G_0)\subseteq \mathrm{Hom}_\sk^{G^0}(T,G^0)$ is finite. As our $G^0$ is Abelian,  we have  $\mathrm{Hom}_\sk^{G^0}(T,G^0)=\mathrm{Hom}_\sk(T/G^0, G^0)$ (cf. \cite[Lem.~4.1]{Brion11}). As by construction $T/G_0$ is an $G(\sk)$-torsor over the smooth curve  $U$,  $T/G^0$ is itself smooth too. In particular, $\mathrm{Hom}_\sk(T/G^0, G^0)$ is the trivial group and we are done.
\end{proof}
\begin{Cor}\label{Cor: Bt definition}
The group  $\et(T/U)$ is infinite if and only if the group
$$
\Bt(T/U):=\mathrm{Im}(\nu_\sharp: \et(T/U)\to \mathrm{Aut}_\sk(U)) \subseteq \mathrm{Aut}_\sk(U)
$$
is infinite. 
\end{Cor}
 The main result of this section is the following:
\begin{Thm}\label{Thm: main on torsors}
Given a pair $(G,\nu:T\to U)$ satisfying the conditions (i)-(iii) in the beginning of  this section, there is a subgroup scheme $G'\subseteq G$, a $G'$-torsor $\nu': T'\to U$ satisfying:
\begin{enumerate}[(a)]
\item $T$ reduces to $T'$ (cf. Definition~\ref{Def: reduction of structure group});

\item $T'$ is irreducible and smooth over $\sk$;

\item $\et(T'/U)$ (equivalently $\Bt(T'/U)$ by Corollary~\ref{Cor: Bt definition}) is also infinite.

\item let $T'\subseteq C'$ be the normal compactification, then the $G'$-action on $T'$ extends to $C'$.
\end{enumerate}
\end{Thm}

\begin{Rem}\label{Rem: main on torsors}
(1) We recall that a torsor under a finite group scheme is called \emph{Nori-reduced} if it can not non-trivially reduce any more (cf. \cite[Def.~5.10]{Niels-Vistoli15}) and every such a torsor admits a Nori-reduced torsor associated. The torsor $T'$ that $T$ reduces to in the above theorem is exactly the Nori-reduced torsor associated to $\nu: T\to U$. 

(2) In general, without the infiniteness assumption on $\et(T/U)$,  the Nori-reduced torsor associated is not necessarily reduced. 
\end{Rem}
This theorem will be proved in \S~\ref{Subsec: proof of thm-torsor main} based on the assumption of Theorem~\ref{Thm: main on reduction of torsors}. The assumed Theorem~\ref{Thm: main on reduction of torsors} will be proved in \S~\ref{Subsec: proof of a key theorem in torsor}. Finally, in \S~\ref{Subsec: classification of torsors},  we will give a complete classification of the pairs $(G', \nu': T'\to U)$ mentioned in Theorem~\ref{Thm: main on torsors}.
 
\subsection{Proof of Theorem~\ref{Thm: main on torsors}}\label{Subsec: proof of thm-torsor main}
We shall make several reductions to prove Theorem~\ref{Thm: main on torsors}. Let $(G,\nu: T\to U)$ be as in Theorem~\ref{Thm: main on torsors}. We remind that there is a canonical semi-direct decomposition $G=G^0\rtimes_\sk G(\sk)$ and $G^0$ is Abelian by assumption. First let $M:=T/G^0$, then $\pi: T\to M$ is a $G^0$-torsor (cf. Theorem~\ref{Thm: quotient by gp}(2)) and $\gamma: M\to U$ is a $G(\sk)$-torsor as $G^0$ is normal in $G$.
$$
\xymatrix{
T\ar[rrrrrr]^\pi_{{\text{$G^0$-torsor}}}\ar[rrrd]^{\nu}_{\text{$G$-torsor}} &&&&&& M=T/G^0\ar[llld]_{\gamma}^{\text{\,\, $G(\sk)$-torsor}}\\
&&& U
}
$$

\begin{Lem}
With respect to their torsor structures, the groups $\et(T/M)$ and $\et(M/U)$ are both infinite.
\end{Lem}
\begin{proof}
For any $(\tau,\sigma)\in \et(T/U)$ (cf. Definition~\ref{Def: equivariant automorphism}), the isomorphism $\tau: T\to T$ is commutative with the $G$-action on $T$ by definition. So there induces an isomorphism $\tau_*: M\to M$ commutative with the $G(\sk)$-action on it as follows:
$$
\xymatrix{
T\ar[d]^\tau \ar[rr]^\pi && M \ar[rr]^{\gamma} \ar[d]^{\tau_*}&& U\ar[d]^\sigma \\
T \ar[rr]^\pi && M \ar[rr]^{\gamma} && U
}
$$ In other words, $(\tau,\tau_*), (\tau_*,\sigma)$ are contained in $\et(T/M)$ and $\et(M/U)$ respectively. Therefore the groups $\et(T/M)$ and $\et(M/U)$ are both infinite.
\end{proof}

Then we note that $M$, as an \'etale Galois cover of $U$ with Galois group $G(\sk)$ (cf. Remark~\ref{Rem: on torsors}) is smooth over $\sk$ as $U$ is so. Fixing an arbitrary irreducible component $M_1$ of $M$, then $M_1$ is a reduction of $M$ by the group extension $H=\{g\in G(\sk)| g(T_2)=T_2\}\subseteq G(\sk)$ (cf. Example~\ref{Exa: typical example of extension and reduction of torsors}). In other words, let $G_1=G^0\rtimes H\subseteq G^0\rtimes G(\sk)=G$ be the associated subgroup scheme, then $T$ first reduces to a $G_1$-torsor $T_1:=\pi^{-1}(M_1)$ so that $T_1$ is irreducible. 
$$
\xymatrix{
T_1=\pi^{-1}(M_1) \ar[rrrd]^{\nu_1:=\nu|_{T_1}}_{\text{$G_1'$-torsor}} \ar[rrrrrr]^{\pi|_{T_1}}_{\text{$G^0$-torsor}} &&&&&& M_1 \ar[llld]_{\gamma_1:=\gamma|_{M_1}}^{\,\, \,\,\,\text{$G_1'(\sk)$-torsor}}\\
&&& U}
$$
 
 \begin{Lem}\label{Lem: local infinite local}
 The group $\et(T_1/U)$ is also infinite.
 \end{Lem}
 \begin{proof}
In fact, given any $(\tau,\sigma)\in \et(T/U)$, if $\tau(T_1)=T_1$, then its restriction on $T_1$ induces an element $(\tau|_{T_1}, \sigma)\in \et(T_1/U)$. As $T, T_1$ are homeomorphic to $M,M_1$ respectively, $T_1$ is an irreducible component of $T$. Since $T$ has only finitely many irreducible components, the subgroup $\{(\tau,\sigma)| \tau(T_1)=T_1\}\subseteq  \et(T/U)$ has a finite index.  We are done.
 \end{proof} 
So similarly, the groups $\et(T_1/M_1)$ and $\et(M_1/U)$ are both infinite too. Up to now, we have reduced $T$ into an irreducible torsor $T_1$. Then we claim the following theorem, which will be proved in \S~\ref{Subsec: proof of a key theorem in torsor}.
\begin{Thm}\label{Thm: main on reduction of torsors}
Suppose $\Lambda$ is an infinitesimal Abelian group scheme, $\nu: X\to Y$ is a $\Lambda$-torsor over an smooth irreducible curve $Y$ with infinite $\et(X/Y)$ (equivalently, infinite $\Bt(X/Y)$). Then $X$ reduces to a torosr $X'$ under a subgroup scheme $\Lambda'\subseteq \Lambda$ satisfying:
\begin{enumerate}[(i).]
\item $X'$ is smooth over $\sk$;

\item $\et(X'/Y)$ (equivalently, $\Bt(X'/Y)$) is still infinite.
\end{enumerate}
\end{Thm} 
Assuming this theorem at the moment, by the reduction we have an identification  $(X'\times_\sk \Lambda)/\Lambda'=X$ (cf. Definition~\ref{Cor: extension of infinitesimal automorphism}) inducing a canonical isomorphism $X'=X_{\mathrm{red}}$ as follows
$$
\xymatrix{
X'\ar@{=}[rr]   && ((X'\times_Y \Lambda)/\Lambda')_{\mathrm{red}} \ar@{^{(}->}[d]\ar@{=}[rr]&& X_{\mathrm{red}}\ar@{^{(}->}[d]\\
                 &&  (X'\times_Y \Lambda)/\Lambda' \ar@{=}[rr] && X 
}
$$In other words, the above theorem can be restated as follows:

\begin{Thm}\label{Thm: restatement}
Suppose $\Lambda$ is an infinitesimal Abelian group scheme, $\nu: X\to Y$ is a $\Lambda$-torsor over an smooth, irreducible curve $Y$ with infinite $\Bt(X/Y)$. Then 
\begin{enumerate}[(i).]
\item $X'=X_{\mathrm{red}}$ is smooth over $\sk$;

\item there is a subgroup scheme $\Lambda'\subseteq \Lambda$ preserving $X'$ and with this $\Lambda'$-action  $X'$ is a $\Lambda'$-torsor over $Y$.
\end{enumerate}
By this description, it is clear that $\Bt(X'/Y)$ contains $\Bt(X/Y)$ as a subgroup.
\end{Thm}
Now take $\Lambda=G^0, X=T_1, Y=M_1$, by Theorem~\ref{Thm: restatement}, we obtain a subgroup $\Lambda'\subseteq G^0$ acting on the smooth $\sk$-variety $T':=(T_1)_{\mathrm{red}}$ making $T'$ a $\Lambda'$-torsor over $M_1$. Namely $T_1$, as a $G^0$-torsor over $M_1$ reduces to $T'$.  But note our target is to prove that $T_1$, as a $G_1$-torsor over $U$,  reduces to $T'$. So we still have to show that the $G_1(\sk)$-action is preserved by the reduction $T_1\to T'$ at the same time. Namely, we need to prove that
\begin{itemize}
\item the $G_1(\sk)$-action on $T_1$ preserves $T'$;

\item the subgroup scheme $\Lambda'$ is preserved by the conjugation action of $G_1'(\sk)$.
\end{itemize}
As $G_1'(\sk)$ is an \'etale group scheme, it clearly preserves $T'=(T_1)_{\mathrm{red}}$. So it remains to prove the second condition.  
\begin{Lem}
 The subgroup scheme $\Lambda'\subseteq G^0$  is preserved by the (conjugation-)action of $G'_1(\sk)$. 
\end{Lem}
\begin{proof}
Note the subgroup scheme $\Lambda'$, is uniquely determined by $X'$ as its normalizer in $G_1'$. In fact, as $T'$ is a $\Lambda'$-torsor over $M_1$,  for any  $\sk$-scheme $W$ with $X'(W)\neq \emptyset$, we have $$\Lambda'(W)=\{h\in G_1'(W)| h \cdot X'(W)=X'(W)\}\subseteq G_1'(W).$$ Now giving an arbitrary $g\in G_1'(\sk)$, the map $g: X\to X$ is an isomorphism of $G'_1$-torsors preserving $X'$ by construction. So
\begin{align*}
 g\Lambda'(W)g^{-1}&=\{h\in G_1'(W)| g^{-1}hg\cdot X'(W)=X'(W)\}\\
                   &=\{h\in G_1'(W)| h g\cdot X'(W)=g\cdot X'(W)\}\\
                   &=\{h\in G_1'(W)| h \cdot X'(W)=X'(W)\}=\Lambda'(W)
\end{align*} Here note as $g(X')=X'$ we have $g\cdot X'(W)=X'(W)$, hence $g\Lambda'g^{-1}=\Lambda'$ as desired.
\end{proof}
In other words, the smooth irreducible curve $T'$ is a $G'=\Lambda'\rtimes G_1(\sk)$-torsor over $U$ which the $G_1$-torsor $T_1$ reduces to.    Note from the construction $T'=(T_1)_{\mathrm{red}}$, any $(\tau,\sigma)\in \et(T_1/U)$ must restricts to an element in $\et(T'/U)$. So $\Bt(T'/U)$ contains $\Bt(T_1/U)$ as a subgroup and hence  $\et(T'/U)$ is infinite (cf. Lemma~\ref{Lem: local infinite local}). Finally, the next lemma completes the proof of the last assertion of Theorem~\ref{Thm: main on torsors}.
\begin{Lem}\label{Lem: extension the action}
Let  $T'\subseteq C'$ be the normal compactification of $T'$. Then the $G'$-action on $T'$ extends to $C'$.
\end{Lem}
\begin{proof}
Note that in this case, we have an infinite subgroup $$H:=\mathrm{Im}(p_1: \et(T'/U)\to \mathrm{Aut}_\sk(T'),\,\,\, (\tau,\sigma)\mapsto \tau)\subseteq \mathrm{Aut}_\sk(T')$$ commutes with the $G'$-action on $T'$ by definition. In particular, $\mathrm{Aut}_\sk(T')$ is infinite. As $T'$ is smooth irreducible, $T'$ is one of the four curves: an elliptic curve, $\p, \A$ or $\Aa$. There is nothing to prove when $T'$ is projective. So it remains to prove when $T'=\A$ or $\Aa$. Note in this case, with the group $H$ in mind, our conclusion follows from Corollary~\ref{Cor: extension of infinitesimal automorphism} immediately.
\end{proof}

In conclusion, assuming Theorem~\ref{Thm: main on reduction of torsors}, we complete the proof of Theorem~\ref{Thm: main on torsors}.

\subsection{Infinitesimal torsors with infinite equivariant automorphism group}\label{Subsec: proof of a key theorem in torsor}
We now prove Theorem~\ref{Thm: main on reduction of torsors}.  Recall that $\Lambda$ is an infinitesimal group scheme over $\sk$, $\nu: X\to Y$ is a $\Lambda$-torsor with infinite $\et(X/Y)$ (or equivalently $\Bt(X/Y)$) and $Y$ is a smooth irreducible curve over $\sk$. Note as $\Bt(X/Y)$ is a subgroup of $\mathrm{Aut}_\sk(Y)$, this curve $Y$ can only be one of the following: an elliptic curve, $\p,\A$ and $\Aa$.

Since $\Lambda$ is finite and commutative over $\sk$ this time, the (fppf-)$\Lambda$-torsors of is classified by the Abelian group $H^1_{\mathrm{fl}}(Y,\Lambda)$:
$$\{\text{(fppf-)$\Lambda$-torsor over $Y$}\}\stackrel{1-1}{\Longleftrightarrow} H^1_{\mathrm{fl}}(Y,\underline{\Lambda}), \,\,\, X\mapsto o_X\in H^1_{\mathrm{fl}}(Y,\underline{\Lambda}).$$
Note that there is a natural $\mathrm{Aut}_\sk(Y)$-action on the group $H^1_{\mathrm{fl}}(Y,\underline{\Lambda})$,  we can then reinterpret $\Bt(X/Y)$ as follows. 
\begin{Lem}\label{Lem: B with stab}
With respect to the natural $\mathrm{Aut}_\sk(Y)$-action, we have $$\Bt(X/Y)=\mathrm{Stab}_{o_{X}}\subseteq \mathrm{Aut}_\sk(Y).$$
\end{Lem}
\begin{proof}
By definition, $\sigma\in  \mathrm{Aut}_\sk(Y)$ is in the stabilizer of $o_X$ if and only if there is an isomorphism $\theta: X\simeq X\times_{Y,\sigma} Y$ commutative with the $\Lambda$-action. The existence of such $\theta$ is the same as the existence of  $\tau:=p_1\circ \theta: X\to X$ making $(\tau,\sigma)\in \et(X/Y)$ by definition.   We are done.
\end{proof}

Next, recall that for a subgroup scheme $\Lambda'\subseteq \Lambda$, there is an $\mathrm{Aut}_\sk(Y)$-equivariant homomorphism
$$\Phi_{\Lambda,\Lambda'}: H^1_{\mathrm{fl}}(Y,\underline{\Lambda'})\to H^1_{\mathrm{fl}}(Y,\underline{\Lambda}).$$ Then a reduction of structure group on a $\Lambda$-torsor $X$ to $\Lambda'$ (cf. Definition~\ref{Def: reduction of structure group}) is nothing but a $\Lambda'$-torosr $X'$ on $Y$ satisfying: $o_X=\Phi_{\Lambda,\Lambda'}(o_{X'})$. 

\begin{Lem}\label{Lem: local in torsors with}
Let $\Lambda'\subseteq \Lambda$ be a subgroup scheme, then  the natural map $\Phi_{\Lambda,\Lambda'}: H^1_{\mathrm{fl}}(Y,\underline{\Lambda'})\to H^1_{\mathrm{fl}}(Y,\underline{\Lambda})$ is injective. In particular, suppose $X$ descends to $X'$, then $\Bt(X'/Y)= \Bt(X/Y)$.
\end{Lem}
\begin{proof}
We have an exact sequence in the fppf topology:
$$0\to \underline{\Lambda'} \to \underline{\Lambda}\to \underline{\Lambda/\Lambda'}\to 0.$$ And we therefore obtain the following long exact sequence:
$$\Lambda/\Lambda'(Y)\to H^1_{\mathrm{fl}}(Y,\underline{\Lambda'})\stackrel{\Phi_{\Lambda,\Lambda'}}{\longrightarrow} H^1_{\mathrm{fl}}(Y,\underline{\Lambda}).$$ As $\Lambda/\Lambda'$ is finite and $Y$ is reduced, we have  $\Lambda/\Lambda'(Y)$ is trivial since $\Lambda$ is infinitesimal.  The in particular part follows from the previous lemma since $\phi_{\Lambda,\Lambda'}$ is equivariant with respect to the $\mathrm{Aut}_\sk(Y)$-action.
\end{proof}

\begin{proof}[Proof of Theorem~\ref{Thm: main on reduction of torsors}]
We shall prove by induction on $r(\Lambda):=\mathrm{rk}(\Lambda)$.

\medskip
{\bf Starting stage: $r(\Lambda)=1$.}

In this case we have $\Lambda=\mu_p$ or $\alpha_p$. 
Using Kummar theory, we have the following two exact sequences:
\begin{equation}\label{eq: Kummar}
H^0(Y,\cO_Y)^*\stackrel{^p}{\to} H^0(Y,\cO_Y)^*\to H_{\mathrm{fl}}^1(Y,\mu_p)\to \mathbf{Pic}(Y)\stackrel{p}{\to} \mathbf{Pic}(Y). 
\end{equation}
\begin{equation}\label{eq: A-R}
H^0(Y,\cO_Y)\stackrel{  F^*_{Y/\sk}}{\to} H^0(Y,\cO_Y)\to H_{\mathrm{fl}}^1(Y,\alpha_p)\to H^1(Y,\cO_Y)\stackrel{ F^*_{Y/\sk}}{\longrightarrow} H^1(Y,\cO_Y).
\end{equation}
\begin{enumerate}[(A.)]
\item When $Y=E$ is an elliptic curve. Let $X_0:=X_{\mathrm{red}}$. Then the map $\rho: X_0\to E$ has two possibilities:
either $\deg(\rho)=1$ or $\deg(\rho)=p$ since $\rho$ is a homeomorphism. 

If $\deg(\rho)=1$,  $\rho: X_0\to E$ is an isomorphism and thus the torsor $\nu:X\to E$ admits a section. So $X$ is trivial and then reduces to the trivial torsor $X'=E$ under the identity group $\Lambda'=e\subseteq \Lambda$.

If $\deg(\rho)=p$, $X=X_0$ is reduced by flatness. Suppose $X$ is not smooth and $x_1,x_2,\cdots, x_n$ are its singularities, then the finite set $y_i=\nu(x_i)\in E$ is preserved by any $\sigma\in \Bt(X/Y)$. Hence $\Bt(X/Y)$ must be finite, a contradiction. 

\item When $Y=\p$. We have $H_{\mathrm{fl}}^1(Y,\mu_p)=H_{\mathrm{fl}}^1(Y,\alpha_p)=0$ by (\ref{eq: Kummar}) and (\ref{eq: A-R}). We are done. 

\item When $Y=\A=\Spec(\sk[t])$. If $\Lambda=\mu_p$, we have $H_{\mathrm{fl}}^1(Y,\mu_p)=0$ by (\ref{eq: Kummar}) and we are done. If $\Lambda=\alpha_p$,  by (\ref{eq: A-R}),  $H^1_{\mathrm{fl}}(Y,\alpha_p)=\sk[t]/\sk[t^p]$. By construction, for each $\overline{f}(t)\in \sk[t]/\sk[t^p]$, the torsor associated to it   is nothing but $X=\Spec(\sk[t,x]/(x^p-f(t)))$.   Now, the automorphism group $\mathrm{Aut}_\sk(\A)=\{\phi_{a,b}: t\mapsto at+b, a\in \sk^*,b\in \sk\}$ acts on $\sk[t]/\sk[t^p] $ by mapping $$ \phi_{a,b}\cdot  \overline{f}(t)= \overline{f}(at+b).$$  So $\phi_{a,b}$ stabilize $\overline{f}(t) \in \sk[t]/(\sk[t])^p$ if and only if $f(t)-f(at+b)\in \sk[t^p]$. An easy calculation shows that only $\overline{f}(t)$ with $f=\lambda t, \lambda\in \sk$ has an infinite stabilizer group. In other words, except for the trivial torsor ($f=0$), other $\alpha_p$-torsors $X$ with infinite $\Bt(X/\A)$ are given as $X:=\Spec(\sk[x,t]/(x^p-\lambda t), \lambda\in \sk^*$, which is smooth over $\sk$.

\item When $Y=\Aa=\Spec(\sk[t,t^{-1}])$. If $\Lambda=\mu_p$, we have by (\ref{eq: Kummar}) $$H^1_{\mathrm{fl}}(Y,\mu_p)=\{1,t,t^2, \cdots, t^{p-1}\}.$$ Except for the trivial torsor represented by $1$, for each $t^i,i=1,\cdots p-1$, the associated torsor   $X:=\Spec(\sk[t,t^{-1},x]/(x^p-t^i))$ is smooth and connected. 

If  $\Lambda=\alpha_p$, we have by (\ref{eq: A-R}), $H^1_{\mathrm{fl}}(Y,\alpha_p)$ is the space  $\sk[t,t^{-1}] /\sk[t^p,t^{-p}]$. Following the exact sequence (\ref{Ext: aut Gm}), the stabilizer of $\overline{f}(t)\in \sk[t,t^{-1}] /\sk[t^p,t^{-p}]$ is infinite if and only if $\overline{f}(t)$ is fixed by infinitely many $\lambda\in \sk^*$.  In other words, $f(\lambda t)-f(t)\in \sk[t^p,t^{-p}]$ for infinitely many $\lambda\in \sk^*$.  An easy calculation shows that only $\overline{f}=0$ satisfying this property and we are done.
\end{enumerate}

\medskip
{\bf Inductive process.}

Now assume $r(\Lambda)\ge 2$. We can first choose an arbitrary  subgroup scheme $ \Xi\subseteq \Lambda$ of rank $1$. Denote by $\overline{X}:= X/\Xi$, the $\Lambda/\Xi$-torsor on $Y$. By inductive assumption, we have 
\begin{itemize}
\item either $\overline{X}$ descends to some $\Lambda'/\Xi$-torsor for a proper subgroup scheme $\Lambda'\subseteq \Lambda$ containing $\Xi$;

\item or  $\overline{X}$ is smooth connected.
\end{itemize}
In the first case,  we have the following commutative diagram of exact sequence:
$$
\xymatrix{
0\ar[r] & \Xi \ar[r]  \ar@{=}[d] & \Lambda' \ar[r] \ar@{^{(}->}[d] & \Lambda'/\Xi \ar[d]\ar[r] & 0 \\
0\ar[r] & \Xi \ar[r]         & \Lambda  \ar[r]                 &\Lambda/\Xi \ar[r] &0
}
$$
Hence, we derive the following commutative diagram.
$$
\xymatrix{
0\ar[r] & H^1(Y, \Xi) \ar[r]  \ar@{=}[d] & H^1(Y, \Lambda') \ar[r] \ar@{^{(}->}[d] &H^1(Y, \Lambda'/\Xi) \ar@{^{(}->}[d]\ar[r] & H^2(Y,\Xi) \ar@{=}[d]\\
0\ar[r] & H^1(Y, \Xi) \ar[r]         & H^1(Y, \Lambda)  \ar[r]^\pi                 &H^1(Y, \Lambda/\Xi) \ar[r] & H^2(Y,\Xi)
}
$$Here the middle vertical arrows are injective by Lemma~\ref{Lem: local in torsors with}. 
By diagram chasing, it is not difficult to see that $o_X\in  H^1(Y, \Lambda)$ this time descends to some $o_{X'}\in H^1(Y, \Lambda')$ since $o_{\overline{X}}=\pi(o_X)$ is so.  Then we are done by the inductive assumption as $r(\Lambda')<r(\Lambda)$.

The next case, where $\overline{X}$ is smooth, we consider the map $\rho: X_0:=X_{\mathrm{red}}\to \overline{X}$. We have two possibilities: $\deg(\rho)=1$ or $p$ since $\rho$ is a homeomorphism of degree at most $p$.

First assume $\deg(\rho)=1$. As $\overline{X}$ is smooth, $\rho$ is an isomorphism. In other words, the $\Xi$-torsor $X\to \overline{X}$ admits a section and hence is trivial. So, the $\Xi$-action map $\iota: \Xi\times_\sk X_0\to X$ is an isomorphism. Namely, we can identify $X=\Xi\times_\sk X_0$ so that the subgroup $\Xi$ acts only on the first factor by transaltions. Then the group action of $\Lambda$ induces gives a natural morphism $$\theta: \Lambda\times_\sk X_0 \to X\stackrel{\iota^{-1}}{\to} \Xi\times_\sk X_0\stackrel{p_1}{\to} \Xi,$$  inducing a map $$X_0\mapsto \underline{\mathrm{Hom}}_\sk(\Lambda,\Xi)$$
\begin{Lem}
The map $\theta$ is constant.
\end{Lem}
\begin{proof}
Since both $\Lambda,\Xi$ are  proper over $\sk$, $\underline{\mathrm{Hom}}_\sk(\Lambda,\Xi)$ is represented by a scheme $\mathbf{Hom}_\sk(\Lambda,\Xi)$ locally of finite type over $\sk$ (cf. \cite[\S~5.6.2]{FGA-explained}). By construction, for any $W\in\mathbf{Sch}/\sk$ and $W$-points $\lambda\in \Lambda(W), x_0\in X_0(W)$, its image $\theta(\lambda,x_0)=\xi\in \Xi(W)$ is such that there is a $x_0'\in X_0(W)$ with $\lambda\cdot x_0=\xi\cdot x_0'$. Now for any $(\tau,\sigma)\in\et(X/Y)$, as $\tau$ automatically preserves $X_0$, we have $$\lambda \cdot (\tau(x_0))=\tau(\lambda\cdot x_0)=\tau(\xi\cdot x_0')=\xi\cdot \tau(x_0').$$ In other words, $\theta(\lambda,\tau(x_0)=\theta(\lambda,x_0)$. So the morphism $X_0\to \mathbf{Hom}_\sk(\Lambda,\Xi)$ is invariant under the isomorphism $\tau|_{X_0}: X_0\to X_0$ for all $(\tau,\sigma)\in \et(X/Y)$. Note for different $\sigma$, the automorphism  $\tau|_{X_0}$ is also different. As $\Bt(X/Y)$ is infinite, there are infinitely many  such automorphism $\tau|_{X_0}$. So the morphism $X_0\to \mathbf{Hom}_\sk(\Lambda,\Xi)$ must be a constant one as $X_0$ is only a curve. 
\end{proof}
From this lemma there is a morphism $\theta: \Lambda\to \Xi$ so that for any $W\in\mathbf{Sch}/Y, g\in \Lambda(W), t\in X_0(\Lambda)$ there is $v(t,g)\in X_0(T)$ so that $g\cdot t=\theta(g)\cdot v(t,g)$. As a result, we have 
\begin{align*}
(g_1g_2)\cdot t&=\theta(g_1g_2)\cdot v(t, g_1g_2) =g_1\cdot (\theta(g_2) \cdot v(t,g_2) )\\
           &=\theta(g_2)\cdot (g_1\cdot  v(t,g_2 ))=\theta(g_2)\cdot (\theta(g_1)\cdot v(v(t,g_2),g_1))\\
           &= \theta(g_1)\theta(g_2)\cdot  v(v(t,g_2),g_1)) 
\end{align*}
The isomorphism $\iota$ then says $$\theta(g_1g_2)=\theta(g_1)\theta(g_2) \,\,\,\text{and}\,\,\,v(t,g_1g_2)=v(v(t,g_2),g_1).$$ In other words, $\theta: \Lambda\to \Xi$ is a homomorphism and the action $$\Lambda\times X_0\to X_0, (g, t)\mapsto v(t,g)$$ is a group action. Let $\Lambda'=\mathrm{Ker}(\theta)$, then $\pi: \Lambda'\stackrel{\simeq}{\to} \Lambda/\Xi$ is an isomorphism. One checks directly the induced $\Lambda'$-action on $X_0$ is identified with the $\Lambda/\Xi$-action on $\overline{X}$.  
$$
\xymatrix{
\Lambda'\times_\sk X_0 \ar[d]^{\pi\times \rho } \ar[rr] && X_0\ar[d]^{\rho}\\
\Lambda/\Xi\times_\sk \overline{X} \ar[rr] && \overline{X}}
 $$ 
Namely, $X_0$ is a $\Lambda'$-torsor reducing $X$. We are done in this case.

Next we assume $\deg(\rho)=p$. Then $X$ is reduced by flatness assumption. Now this time $X\to \overline{X}$ is a non-trivial $\Xi$-torsor. As $r(\Xi)=1$, this $\Xi$-torsor $X$ over $\overline{X}$ can not descend to the identity subgroup, which is the unique subgroup of $\Xi$. It then suffices to prove that this $\Xi$-torsor has an infinite $\Bt(X/\overline{X})$ by our inductive assumption. In fact, for any $(\tau,\sigma)\in \et(X/Y)$, the automorphism $\tau: X\to X$ preserves the $\Lambda$-action by definition, it induces an isomorphism $\overline{\tau}: \overline{X}\to \overline{X}$ so that we have the following commutative diagram:
$$
\xymatrix{
X \ar[rr]^\pi \ar[d]^\tau && \overline{X} \ar[d]^{\overline{\tau}}\ar[rr] &&  Y \ar[d]^\sigma\\
X \ar[rr]^\pi && \overline{X} \ar[rr]&&  Y 
}
$$
In other words, the $\Xi$-torsor $X/X'$ also has an infinite $\Bt(X/X')$. 
 \end{proof}

\subsection{Classification of torsor  with infinite equivaraint automorphism}\label{Subsec: classification of torsors}
Now we  classify the pair $(G',\nu': T'\to U)$ given in Theorem~\ref{Thm: main on torsors}. By changing symbols $(G',T')$ to $(G,T)$, we want in fact to classify the pairs $(T,G)$ so that:
\begin{itemize}
\item $T$ is a smooth irreducible curve over $\sk$;

\item $G$ is a finite group scheme acting freely on $T$;
 
\item the centralizer of $G$ in $\mathrm{Aut}_\sk(T)$ is infinite.
\end{itemize}
With these data, the quotient map $\nu: T\to U:=T/G$ gives a $G$-torsor realization (cf. Themorem~\ref{Thm: quotient by gp}). The group $\et(T/U)$ is identified with the centralizer of $G$ in $\mathrm{Aut}_\sk(T)$ (cf. Example~\ref{Exa: ea and centalizer}). 

The complete classification is as follows:
\begin{Prop}\label{Prop: classification of torsors with infinite Bt}
Let $G$ be a finite group scheme over $\sk$, $T$ be a smooth irreducible curve over $\sk$ admitting a free $G$-action so that the centralizer of $G$ in $\mathrm{Aut}_\sk(T)$ is infinite. Then the pair $(T,G)$ is one of the following:
\begin{enumerate}[(A.)]
\item  $T=E'$ is an elliptic curve. Then subgroup schemes of $\mathbf{Aut}_\sk(E')$ acting freely on $E'$ are contained in the translation subgroups $E' \subseteq \mathbf{Aut}_\sk(E)$ by Proposition~\ref{Prop: classification of subgroup acting freely on E}. So $G$ is a finite subgroup scheme of $E'$ acting  naturally on $T=E'$ by natural translations. This time $U=E'/G$ is an elliptic curve. 

\item   $T=\p$, $G=\{\mathrm{id}\}$. This time $\nu=\mathrm{id}: T=\p \to U=\p$.

\item ($\mathrm{char}.(\sk)=p>0$) $T=\A$.  This time the group $G$ is a subgroup scheme of $\mathbf{LAut}_\sk(\A)=\mathbb{G}_a\rtimes_\sk \mathbb{G}_m$ by Corollary~\ref{Cor: extension of infinitesimal automorphism}.  Then following  Proposition~\ref{Prop: classification of subgroup acting freely on U}, $G$ is a subgroup of the semi-direct summand $\mathbb{G}_a\subseteq \mathbf{LAut}_\sk(\A)$ acting additively on $\A$ (cf. Example~\ref{Exa: example of group actions}). This time $U=T/G\simeq \A$.  

\item $T=\Aa$. This time the group $G$ is a subgroup scheme of $\mathbf{LAut}_\sk(\Aa)=\mathbb{G}_m$ by Corollary~\ref{Cor: extension of infinitesimal automorphism}.  In other words, $G$ is a finite subgroup scheme of $\mathbb{G}_m$-acting multiplicatively on $\Aa$. This time $U=T/G\simeq \A$.   
\end{enumerate}
\end{Prop}
\begin{Rem}\label{Rem: G is abelian in classification}
In the above classification,  $G$ is always Abelian. 
\end{Rem}

\section{On genus one fibrations}\label{Sec: genus one fibration}
The genus one fibration is the most difficult part of this paper. We shall make some preparations in this section. As we need to work over non-algebraically closed field, we make the following conventions on genus one curves and elliptic curves.
\begin{Def}[Genus one curve VS. Elliptic curve] We will fix the following definition in this paper:
\begin{itemize}
\item  a genus one curve $X_K$ over a field $K$ is referred to a smooth projective, geometrically connected curve of genus one.  

\item  an elliptic curve $E_K$ over $K$ is referred to a genus one curve along with a fixed group scheme structure. In particular, $E_K(K)\neq \emptyset$.

\item a genus one ({\it resp.} an elliptic) surface fibration $f:X\to C$  ({\it resp.} $\mu: J\to C$) over $\sk$ is a fibration whose generic fibre  is a genus one ({\it resp.} an elliptic) curve.
\end{itemize}
\end{Def}
There is a connection between  genus one surface fibrations and  elliptic fibrations known as the Jacobian fibration. Let us recall in the following. Let $f:X\to C$ be a surface fibration of genus one over $\sk$. By our convention,  there is an open subset $U\subseteq C$ so that $f_U: X_U:=X\times_C U\to U$ is smooth.  The relative Picard functor $\underline{\mathrm{Pic}}_{X_U/U}$ (cf. \cite[pp.~201, Chapter~8]{BLR90-Neron-Models}) classifying the relative invertible sheaves on $X_U$
is represented by a formally smooth group scheme $\mathbf{Pic}_{X_U/U}$ locally of finite type (cf. \cite[$n^\circ 232$, Thm.~3.1]{FGA} or \cite[pp.~210, Thm.~1]{BLR90-Neron-Models})), called the relative Picard scheme. This relative Picard scheme is a disjoint union  
$$\mathbf{Pic}_{X_U/U}=\coprod \limits_{n\in \Z} \mathbf{Pic}^n_{X_U/U}$$ of subschemes $\mathbf{Pic}^n_{X_U/U}$ classifying relative invertible sheaves on degree $n, n\in \Z$. These components $\mathbf{Pic}^n_{X_U/U}$ are smooth and projective over $U$.  The identity component $J_U:=\mathbf{Pic}^0_{X_U/U}$ is in particular a smooth projective group scheme over $U$ of relative dimension $g=1$, or equivalently $J_U \to U$ is a relatively minimal elliptic fibration over $U$. We then denote by $\mu: J\to C$ its relative minimal model over $C$. 

\begin{Def}[Jacobian fibration]\label{Def: Jacobian fibration}
This model $\mu: J\to C$ is called the Jacobian fibration of $f:X\to C$.
\end{Def} 
Let $K:=K(C)$  be the function field of $C$, $\eta:=\Spec(K)\hookrightarrow C$ be the generic point of $C$. Then by construction the generic fibre $J_\eta$ of the Jacobian fibration  $\mu: J\to C$ is the elliptic curve $\mathbf{Pic}^0_{X_\eta/K}$. In other words, the Jacobian fibration of a genus one fibration is first replacing the generic fibre by  its Jacobian ({\it i.e.}, $\mathbf{Pic}^0$) and then taking relatively minimal model.

Noticing that every degree one invertible sheaf on a genus one curve is the invertible sheaf associated to a unique degree one effective divisor on it,  we can then identify $X_U$ canonically with $\mathbf{Pic}^1_{X_U/U}$. So the group structure of $\mathbf{Pic}_{X_U/U}$  gives a natural group $J_U$-action on $X_U$ 
$$
\xymatrix{
\mathbf{Pic}^0_{X_U/U} \times_U\mathbf{Pic}^1_{X_U/U} \ar@{=}[d] \ar[rr]  &&\mathbf{Pic}^1_{X_U/U} \ar@{=}[d]\\
J_U \times_U  X_U\ar[rr] &&X_U
}
$$ With this action, $X_U$ is regarded naturally as an (\'etale-)torsor under $J_U$. Similarly, the schemes $\mathbf{Pic}^n_{X_U/U}$ also admits natural (\'etale-)torsor structure under $J_U=\mathbf{Pic}^0_{X_U/U}$. It is well known that \'etale torsors of $J_U$ are classified by the group $H^1_{\acute{e}t}(U,\underline{J_U})$.  Let $o_{X_U}\in H^1_{\acute{e}t}(U,\underline{J_U})$ be the element corresponding to $X_U$. Then we can see that the torsor corresponding to $n\cdot o_{X_U}\in H^1_{\acute{e}t}(U,\underline{J_U})$ is exactly $\mathbf{Pic}^n_{X_U/U}$.

\begin{Prop}\label{Prop: torsor of genus one is torsion}
The element $o_{X_U}\in H^1_{\acute{e}t}(U,\underline{J_U})$ is a torsion element. In fact, for any irreducible horizontal $D$, we have  $[D:C]\cdot o_{X_U}=0$.
\end{Prop}
\begin{proof}
Let $n:=[D:C]$. The  divisor $D$ gives an element $\cO_X(D)|_{X_U} \in \mathbf{Pic}^n_{X_U/U}(U)$. In other words, the torsor $\mathbf{Pic}^n_{X_U/U}$ is trivial and we are done.
\end{proof} 

Then we study the relative automorphism group $\mathrm{Aut}_C(X)$. When $f$ is relatively minimal we have  $\mathrm{Aut}_C(X)=\mathrm{Aut}_K(X_\eta)$. Denote by $\overline{K}:=K^{\mathrm{sep}}$ and $\overline{\eta}:=\Spec(\overline{K})$ the geometric generic point of $C$.
\begin{Prop}
We have an exact sequence: 
\begin{equation}
0\to J_\eta(K)\stackrel{\varrho}{\to} \mathrm{Aut}_K(X_\eta) \stackrel{\nu}{\to} \mathrm{Aut}^{\mathrm{gp}}_{K}(J_{\eta})
\end{equation}
\end{Prop}
\begin{proof}
First note that the group action  of $J_\eta$ on $X_\eta$ gives an injective homomorphism $\varrho: J_\eta(K)\to \mathrm{Aut}_K(X_\eta)$ naturally.   Then given an $K$-automorphism $\sigma: X_\eta\to X_\eta$, there induces a group homomorphism $\sigma^*: \mathbf{Pic}^0_{X_\eta/K}=J_\eta \to \mathbf{Pic}^0_{X_\eta/K}=J_\eta$. The second map  $\nu$ is given by $\sigma\mapsto \sigma^*$. One can clearly extend the two maps above to $X_{\overline{K}}$.

Over $\overline{K}$, we know $X_{\overline{\eta}}$ is isomorphic to $J_{\overline{\eta}}$.  It is well known that the automorphism group  of an elliptic curve is generated by translations and group automorphism. Therefore  the following  sequence of $\Gamma:=\mathrm{Gal}(\overline{K}/K)$-groups obtained similarly to that over $K$ is exact:
\begin{equation*}
0\to J_\eta(\overline{K})\stackrel{\overline{\varrho}}{\to} \mathrm{Aut}_{\overline{K}}(X_{\overline{\eta}})\stackrel{\overline{\nu}}{\to} \mathrm{Aut}^\mathrm{gp}_{\overline{K}}(J_{\overline{\eta}})\to 1
\end{equation*}
Since Galois descent is effective on morphisms, we have $$J_\eta(K)=J_\eta(\overline{K})^{\Gamma}, \mathrm{Aut}_K(X_\eta)=\mathrm{Aut}_{\overline{K}}(X_{\overline{\eta}})^{\Gamma} \,\,\text{and} \,\,\mathrm{Aut}_{K}^{\mathrm{gp}}(J_\eta)= (\mathrm{Aut}^\mathrm{gp}_{\overline{K}}(J_{\overline{\eta}}))^{\Gamma}.$$ Then the desired exact sequence follows immediately.
\end{proof}

\begin{Cor}\label{Cor: ext of auto of genus one}
If $f:X\to C$ is relatively minimal, we have the following exact sequence:
\begin{equation}\label{ext: aut of genus one}
0\to J(C)\to \mathrm{Aut}_C(X) \to \mathrm{Aut}^{\mathrm{gp}}_{K}(J_\eta).
\end{equation}
\end{Cor}

Finally, we recall the following result on the rigidity of subgroups of an elliptic curve.
\begin{Prop}\label{Prop: rigidity of subgroup}
Let $E$ be an elliptic curve over $\sk$, then its finite subgroup schemes are rigid. In other words, if $K$ is a field extension of $\sk$ and $M_K$ is a finite subgroup scheme of $E_K:=E\times_\sk K$, then there is a finite subgroup scheme $\Lambda$ of $E$ so that $M_K=\Lambda\times_\sk K$.
\end{Prop}

\begin{proof}
Clearly, the finite \'etale subgroup schemes of $E$ are rigid. So it remains to assume $M_K$ is infinitesimal. If $\mathrm{char}. (\sk)=0$, then $M_K$ is trivial by Chevalley's celebrated theorem and we are done. Otherwise, we assume $M_K$ is not trivial. Then $M_K$ must contain the kernel of the relative Frobenius homomorphism of $E_K$, which descends to that kernel over $\sk$. We are done then by modulo this kernel and take inductions on the length of $M_K$. 
\end{proof}
\begin{Rem}
This proposition fails in positive characteristics if $E$ is replaced by an Abelian variety of dimension at least two. See \cite{Moret-Bailly}.
\end{Rem}

\section{Surface fibration with a large relative automorphism group}\label{Sec:Fibration with large relative automorphism group}
In this section, we study the structure of $f: X\to C$ of fibre genus $g\ge 1$ with $\mathrm{Aut}_C(X)$ infinite. 

\subsection{A criterion for large relative automorphism group}
Let $K:=K(C)$ be the function field of $C$, $\eta:=\Spec(K)\hookrightarrow C$ the generic point of $C$ and $X_\eta$ the generic fibre of $f:X\to C$. Then since $f$ is relatively minimal, we have $\mathrm{Aut}_C(X)\simeq \mathrm{Aut}_K (X_\eta)$. 

\begin{Thm}\label{Thm: main1}
The group $\mathrm{Aut}_C(X)$ is infinite ({\it resp.} infinitely generated) if and only if $g=1$ and the Jacobian fibration $g: J\to C$ of $f$ is has infinitely many sections ({\it resp.} the Jacobian fibration is trivial).
\end{Thm}
\begin{proof}
If $g\ge 2$, then $\mathrm{Aut}_C(X)=\mathrm{Aut}_K (X_\eta)\subseteq \mathrm{Aut}_{\overline{K}}(X_{\overline{\eta}})$ is finite.

If $g=1$, let $\mu: J\to C$ be the Jacobian fibration of $f:X\to C$.  Then we have an exact sequence $(\ref{ext: aut of genus one})$ in Corollary~\ref{Cor: ext of auto of genus one}. Noticing   $\mathrm{Aut}^{\mathrm{gp}}_{K}(J_\eta)\subseteq \mathrm{Aut}^\mathrm{gp}_{\overline{K}}(J_{\overline{\eta}})$  is always finite, $\mathrm{Aut}_C(X)$ is infinite ({\it resp.} infinitely generated) if and only if $J(C)=J_\eta(K)$ is infinite {\it resp.} infinitely generated). So it remains to prove that $J(C)=J_\eta(K)$ is infinitely generated if and only if $J$ is trivial, or in other words $J=E\times_\sk C$ for an elliptic curve $E$. The `if' part of the claim is clear and we now prove the `only if' part.  In fact let $E:=\mathrm{Tr}_{K/\sk}(J_\eta)$ be the $K/\sk$-trace of $J_{\eta}$ (cf. \cite{Conrad06-Chow's-K/k-trace}). By the celebrated theorem of Lang-N\'eron (cf. \cite[Thm.~7.1]{Conrad06-Chow's-K/k-trace}), the group $J_\eta(K)/E(\sk)$ is finitely generated. As a result, $E(\sk)$ is infinite and hence $E$ is an elliptic curve over $\sk$. So the trace map  $\pi: E\times_\sk K\to J_\eta$ is a surjective group homomorphism. Due to Proposition~\ref{Prop: rigidity of subgroup}, the kernel of $\pi$ is defined over  $\sk$ and hence must be zero by the universal property of the $K/\sk$-trace. In other words, $J_\eta=E\times_\sk K$ and thus $J=E\times_\sk C$ is trivial. We are done.
\end{proof}

\subsection{Standard Genus one fibrations with trivial Jacobian}
We first present some typical construction of genus one fibration with trivial Jacobian. 
\begin{Exa}[Standard genus one fibrations with trivial Jacobian]\label{Exa: genus one fibration with trivial Jacobian, p=0}
Let $E$ be an elliptic curve over $\sk$, $\Lambda\subseteq E$ be a finite subgroup scheme and $C'$ be a smooth projective curve admitting a faithful $\Lambda$-action.  Therefore $\Lambda$  acts  diagonally on the product $E\times_\sk C'$. Note this diagonal action is free since the $\Lambda$-action on $E$ is so. Then the fibration $f:=p_2: X:=(E\times_\sk C')/\Lambda\to C:=C'/\Lambda$ is  a relatively minimal genus one fibration with trivial Jacobian. 
$$
\xymatrix{
X\ar[d]^f \ar@{=}[rr] && (E\times_\sk C')/\Gamma \ar[d]^{p_2}\\
C  \ar@{=}[rr] && C'/\Gamma
}
$$
In fact, this $X$ is smooth by Corollary~\ref{Cor: quotient by gp} and as the translations on $E$ induces a trivial action on its $\mathbf{Pic}^0$, the Jacobian fibration of $f$ is trivial.  
\end{Exa}
\begin{Rem}\label{Rem: feature of naive example}
(1). The subgroup scheme $\Lambda\subseteq E$ can be non-\'etale in positive characteristic.

(2). One key feature for these fibration $f:X\to C$ is that it admits another fibration $h=p_1: X=(E\times_\sk C')/\Lambda \to E/\Lambda$ onto an elliptic curve $E/\Lambda$. This fibration $h$ is smooth and all of its closed fibres are isomorphic to $C'$ (cf. Corollary~\ref{Cor: quotient by gp}). 
\end{Rem} 
 
\begin{Thm}\label{Thm: standard genus one fibration with trivial Jacobian}
Example~\ref{Exa: genus one fibration with trivial Jacobian, p=0} exhausts all relatively minimal genus one fibration with trivial Jacobian fibration that is at the same time standard (cf. Definition~\ref{Def: standard fibration}). 
\end{Thm}
\begin{proof}
Let $(E,C',G)$ be a standard realization of this fibration $f: X\to C$ and let $\varrho: G\to \mathbf{Aut}_\sk(E)$ be the associated $G$-action on $E$. It suffices to prove that $\varrho: G\to \mathbf{Aut}_\sk(E)$  factors through the translation subgroup scheme $E\subseteq \mathbf{Aut}_\sk(E)$. In fact, let $G_1:=\varrho^{-1}(E) \subseteq G, N:=G/G_1\subseteq \mathrm{Aut}_\sk^{\mathrm{gp}}(E)$, and $C_1'=C'/G_1$. Then $C_1'$ admits a natural faithful $N$-action so that $C_1'/N=C'/G=C$. As the translations induces trivial action on Jacobian level, the Jacobian fibration of $f$ is now birational equivalent to 
$$p_2: (E\times_\sk C_1')/N\to C=C_1'/N$$
with diagonal $N$-action on $E\times_\sk C_1'$. In particular, the Jacobian fibration is trivial if and only if $N$ is trivial, or equivalently $\varrho$ factors through the translation subgroup scheme $E\subseteq \mathbf{Aut}_\sk(E)$.  We are done.
\end{proof}

After \protect{\cite[\S~VI]{Beauville-book}} and \cite{Serrano96}, which tell  us any isotrivial fibration in characteristic zero is standard, we have the following classification theorem.
\begin{Thm} \label{Thm: genus one fibration with trivial Jacobian, p=0}
In characteristic zero, every relatively minimal genus one fibration with trivial Jacobian is as in Example~\ref{Exa: genus one fibration with trivial Jacobian, p=0}.
\end{Thm}

\subsection{Non-standard examples}\label{Subsec: genus one fibration with trivial Jacobian, p>0}
 Unlike the characteristic zero case, there are non-standard isotrivial genus one fibration with trivial Jacobian. We shall present such an example in the following.
\begin{Exa}\label{Exa: new}
We first choose a supersingular elliptic curve $E$ over $\sk$ and fix an isomorphism  $\alpha_p\simeq \mathrm{Ker}(F_{E/\sk}: E\to E^{(p)})\subseteq E$. Thus, $\alpha_p$ is a subgroup scheme of $E$ and can act freely on $E$ by translations. Then we choose an arbitrary smooth projective curve $C'$ and denote by $C:=C'^{(p)}$. Let $U'\subseteq C'$ be an open subset admitting a free action of $\alpha_p$ and $U:=U'^{(p)}\subseteq C$. Then $\alpha_p$ can act on $E\times_\sk U'$ diagonally as usual. Note this action is free on both factors.

Take the fibration $f_0=p_2: X_0:=(E\times_\sk U')/\alpha_p \to U'/\alpha_p=U'^{(p)}=U$. It can be seen that all closed fibres of $f_0$ are isomorphic to $E$ (cf. Corollary~\ref{Cor: quotient by gp}). So $f_0$ is a smooth  genus one fibration over $U:=U'^{(p)}\subseteq C$. Take $f: X\to C$ to be the relative minimal model of $f_0$. 
$$
\xymatrix{
E\times_\sk U \ar[rr]^{\pi} \ar[d]^{p_2} && X_0= (E\times_\sk U')/\alpha_p\ar@{^{(}->}[rr]^{\,\,\,\,\,\,\,\,\,\,\,\,\,\,\,\,\text{relatively}}_{\,\,\,\,\,\,\,\,\,\,\,\,\,\,\,\,\,\,\,\,\,\,\,\,\text{minimal model}} \ar[d]^{f_0=p_2}&& X \ar[d]^f\\
U' \ar[rr]^{F_{U'/\sk}}                  &&  U'^{(p)}=U'/\alpha_p \ar@{^{(}->}[rr] &&  C
}
$$ We claim this fibration $f$ has a trivial Jacobian. The reason is similar to Example~\ref{Exa: genus one fibration with trivial Jacobian, p=0}: as $\alpha_p$ acts by translations on $E$, the Jacobian fibration of $f_0$ is trivial over $U$ and hence over $C$.

Finally, we mention that the open subset $U'$ always exists. We take $U\subseteq C$ to be any open affine subset admitting a regular function $x$ so that its K\"ahler differential $\mathrm{d}x$ generates $\Omega_{U/\sk}^1$. Then $\cO_{U'}=\cO_U[\sqrt[p]{x}]$. We can define the free $\alpha_p$-action on $U'$ over $U$ functorially by 
$$
(\alpha, \sqrt[p]{x})\mapsto \sqrt[p]{x}+\alpha, \,\,\, \forall\,\,\, T\in \mathrm{Sch}/U, \,\, \text{and} \,\, \alpha\in \alpha_p(T)=\{\alpha\in \cO_T | \alpha^p=0\}.
$$ 
\end{Exa}  
These examples in general are different to Example~\ref{Exa: genus one fibration with trivial Jacobian, p=0}. The key reason is that the $\alpha_p$-action on $U'$ does not necessarily extend to $C'$. In fact, recall by Proposition~\ref{Prop: action by alphap}, the $\alpha_p$-action on $U'$ is given by a regular derivative on $U'$ which a priori, does not necessarily extend to a global vector field on $C'$.  Then we present a concrete example as below.
\begin{Exa}[A concrete example]\label{Exa: concrete}
Taking $\A =\Spec(\sk[x])\subseteq \p=C'$ and the $\alpha_p$-action on $\A$  given by $$(\alpha,  x) \mapsto x+\alpha\cdot x^p, \,\,\,\forall \,\,\, T\in\mathrm{Sch}/\sk, \alpha\in \alpha_p(T)=\{\alpha\in \cO_T | \alpha^p=0\}$$ in Example~\ref{Exa: new}. Then the regular derivative associated to this $\alpha_p$-action by Proposition~\ref{Prop: action by alphap} is $\delta_1=x^p\partial_x \in H^0(\A,\mathcal{T}_{\A/\sk})$ and in particular, the $\alpha_p$-action is free on the open subset $U'=\Aa\subsetneq \A$ where $\delta_1$ is non-vanishing.  Let $\delta_2\in H^0(E,\mathcal{T}_{E/\sk})$ be the non-zero global vector field on $E$ associated to the free $\alpha_p$-action on $E$. Then the derivative $\delta:= 1\otimes \delta_1+ \delta_2\otimes 1$ on $E\times_\sk \A$ is the one associated to the diagonal $\alpha_p$-action on the open subset $E\times_\sk \A$ in Proposition~\ref{Prop: action by alphap}.  We can easily check that the foliation $\cF(\delta)=\cO_{E\times_\sk\A }\cdot \delta$ (cf. the paragraph below Proposition~\ref{Prop: action by alphap}) extends to a \emph{smooth} foliation $\cF$ on the total space $E\times_\sk \p$:
\begin{itemize}
\item $\cF|_{E\times_\sk \A}$ is $\cF(\delta)$ generated by $\delta$;

\item $\cF$ is generated by $D_2=t^{p-2}\delta=1\otimes \partial_t +\delta_2\otimes t^{p-2}$ near $E\times_\sk\infty$, here $t=1/x$ is the local parameter at $\infty$. 
\end{itemize} 
As a result, the surface $X$ is $(E\times_\sk \p)/\cF$ (cf. Theorem.~\ref{Thm: Ekedahl}). $$\xymatrix{
E\times_\sk\p\ar[d]^{p_2} \ar[r]^{\pi\,\,\,\,\,\,\,\,\,\,\,\,} & X=(E\times_\sk \p)/\cF \ar[d]^f & \ar@{_{(}->}[l]  (E\times_\sk \A)/\cF(\delta) \ar@{=}[r]\ar[d]^{p_2}& (E\times_\sk \A)/\alpha_p \ar[d]^{p_2}\\
C'=\p  \ar[r]^{\text{Frobenius map}}                              & C=\p                             &  \ar@{_{(}->}[l]  \A/\cF(\delta_1) \ar@{=}[r]           & \A/\alpha_p
}$$

By construction, we have $c_1(\cF)=p_2^*\cO_{\p}(-(p-2))$ and hence  by (\ref{formula: canonical bundle and foliation}), 
$$\pi^*K_X=K_{E\times_\sk \p}-(p-1)c_1(\cF)=p_2^*\cO_{\p}(p(p-3)).$$ In particular, as soon as $p\ge 3$, $K_X$ is nef. So the general fibre of its Albanese fibration  $h=p_1: X\to E/\alpha_p$ is not  smooth (otherwise  as $X$ is uniruled, the general fibre has to be $\p$ and $X$ is ruled, a contradition). In other words, there is no fibration from $X$ onto an elliptic curve with smooth fibres, ruling out the possibility of $f$ to be standard (cf. Remark~\ref{Rem: feature of naive example}~(2) and Theorem~\ref{Thm: standard genus one fibration with trivial Jacobian}). 

Finally, when $p\ge 5$,  we have $\kappa(X)=1$. In other words,  $f: X\to \p$ is an isotrivial genus one fibration with only two singular fibres over $\p$ satisfying $\kappa(X)=1$. This is never the case in characteristic zero (cf. \S~\ref{Subsec: k(S)=1}).
\end{Exa}

\subsection{Structure on genus one fibration in positive characteristics}\label{Subsubsec: Structure on genus one fibration in positive characteristics}
After Proposition~\ref{Prop: isotrivial=pre standard} in \S~\ref{subsub: no name} saying that isotrivial fibrations of genus one are pre-standard type, we give  a coarse classification on genus one fibrations with trivial Jacobian fibration in positive characteristic.

\begin{Thm}\label{Thm: structure of genus one with trivial Jacobian p>0}
Let $f: X\to C$ be a relatively minimal genus one surface fibration with trivial Jacobian fibration. Then there is an elliptic curve $E$ over $\sk$, a finite subgroup scheme $\Lambda\subseteq E$, an open subset $V\subseteq C$ and a $\Lambda$-torsor $T'$ over $V$ which is itself irreducible and smooth over $\sk$. With these data,  $f: X\to C$ is the relatively minimal model of the fibration $p_2: (T'\times_\sk E)/\Lambda \to V=T'/\Lambda$. Here $\Lambda$ act diagonally on $T'\times_\sk E$ as before.
$$
\xymatrix{
E\times_\sk T' \ar[rr]^{\pi} \ar[d]^{p_2} &&   (T'\times_\sk E)/\Lambda\ar[r]^{\,\,\,\,\,\,\,\,\,\,\,\,\,\,\,\,\,\,\,\,\simeq} \ar[d]^{p_2}& X_V\ar@{^{(}->}[r]\ar[d]^{f_V}  & X \ar[d]^f\\
T' \ar[rr]                  && T'/\Lambda \ar[r]^{\simeq} & V \ar@{^{(}->}[r]&  C
}
$$
\end{Thm}
\begin{proof}
By Proposition~\ref{Prop: isotrivial=pre standard}, we have all the data $(E,\Lambda, T'/V)$ except that a prior, $\Lambda$ is a subgroup scheme of $\mathbf{Aut}_\sk(E)$ rather than the translation subgroup scheme $E$ as desired. Following the same argument in the proof of Theorem~\ref{Thm: standard genus one fibration with trivial Jacobian}, the fact $\Lambda$ is contained in $E$ comes from the triviality assumption of the  Jacobian fibration.
\end{proof}

\section{Surface fibration with a large B-automorphism group, $g\ge 2$}\label{Sec: B infinite, g>1}
In this and the next section, we study the relatively minimal surface fibration $f: X\to C$ over $\sk$ with infinite $\B(X/C)$. In this section, we assume the fibre genus of $f$ is $g\ge 2$. The treatment here is well known and standard.   We take $U\subseteq C$ to be the smooth locus of $f$ and $f_U: X_U:=X\times_CU\to U$   in the following of this section.

\begin{Lem}\label{Lem: basic}
The fibration $f$ is isotrivial. Namely, there is a smooth curve $F$ of genus $g$ isomorphic to any closed fibre of $f_U:X_U\to U$.
\end{Lem}
\begin{proof}
By the definition of $\B(X/C)$, for any $\sigma\in \B(X/C)$ the fibre $f^{-1}(c)$ is isomorphic to $f^{-1}(\sigma(c))$ for any closed point $c\in C$. As $\B(X/C)$ is infinite,  $f$ is isotrivial. 
\end{proof} As $f$ is isotrivial and the functor $\underline{\mathrm{Aut}}_\sk(F)$ is rigid, the functor $$\underline{\mathrm{Iso}}_U(X_U, F\times_\sk U): \mathbf{Sch}/U \to \mathbf{Set}, \,\,\, W\mapsto \mathrm{Iso}_W(X_U\times_U W, F\times_\sk W)$$ is represented by an \'etale Galois  cover $\iota:  \mathbf{Iso}_U(X_U, F\times_\sk U)\to U$ of smooth curves whose Galois group is canonically identified with $\mathrm{Aut}_\sk(F)$. The Galois group $\mathrm{Aut}_\sk(F)$ action on $\mathbf{Iso}_U(X_U, F\times_\sk U)$ is functorially described as follows. For any $\gamma\in \mathrm{Aut}_\sk(F)$ and $\zeta:  X_U\times_U W\stackrel{\sim}{\to} F\times_\sk W \in \underline{\mathrm{Iso}}_U(X_U, F\times_\sk U)(W)$ the element $\gamma\cdot \zeta$ is the following composition of  isomorphisms
$$
\gamma\cdot \zeta: X_U\times_U W\stackrel{\zeta}{\to} F\times_\sk W \stackrel{\gamma\times \mathrm{id}}{\longrightarrow} F\times_\sk W  \in \underline{\mathrm{Iso}}_U(X_U, F\times_\sk U)(W).
$$
To avoid confusions, we  denote by $\gamma^*: \mathbf{Iso}_U(X_U, F\times_\sk U) \to  \mathbf{Iso}_U(X_U, F\times_\sk U)$ the Galois action  for any $\gamma\in \mathrm{Aut}_\sk(F)$. Following this Galois action, taking $$\phi: X_U\times_U\mathbf{Iso}_U(X_U, F\times_\sk U)\to F\times_\sk\mathbf{Iso}_U(X_U, F\times_\sk U) $$ to be the universal isomorphism, we have the   commutative diagram below for any   $\gamma\in \mathrm{Aut}_\sk(F)$:
$$
\xymatrix{
X_U\times_U\mathbf{Iso}_U(X_U, F\times_\sk U)\ar[d]^{\mathrm{id}\times \gamma^*} \ar[rr]^{\phi} && F\times_\sk \mathbf{Iso}_U(X_U, F\times_\sk U)\ar[d]^{\gamma\times \gamma^*}\\ 
X_U\times_U\mathbf{Iso}_U(X_U, F\times_\sk U) \ar[rr]^\phi&& F\times_\sk \mathbf{Iso}_U(X_U, F\times_\sk U).
}
$$  

Let $T:=\mathbf{Iso}_U(X_U, F\times_\sk U)$ and $G:=\mathrm{Aut}_\sk(F)$ for short. 
Then in summation, 
\begin{itemize}
\item $T$ is a $G$-torsor over $U$;

\item $X_U\times_U T$ is canonically isomorphic (via $\phi$) to $F\times_\sk T$

\item under the canonical isomorphism, the Galois action on $X_U\times_UT$ is translated to the diagonal $G$-action on $F\times_\sk T$. In particular, with the diagonal $G$-action on  $F\times_\sk T$, we have the following commutative diagram:
 \begin{equation}\label{equ: product-quotient g>2}
 \xymatrix{
X\ar@{=}[rr]\ar[d]^{f_U} && (X\times_U T)/G \ar[d]^{p_2} \ar[rr]_\simeq^{\phi} && (F\times_\sk T)/G \ar[d]^{p_2}\\
U \ar@{=}[rr] && T/G        \ar@{=}[rr] && T/G              .}
 \end{equation} 
\end{itemize}

\begin{Prop}\label{Prop: infinite Bt g>2}
 The group $\et(T/U)$  (cf. Definition~\ref{Def: equivariant automorphism}) is infinite.
\end{Prop}
\begin{proof}
Given any $(\tau,\sigma)\in \ea(X/C)$, let $X_U^\sigma:=X_U\times_{U,\sigma} U$. Then the isomorphism $$(\tau,f_U): X_U\to X_U=X_U\times_{U,\sigma} U$$
$$ \xymatrix{
 X_U \ar[drr]^{(\tau,f_U)}_{\simeq }\ar@/^1pc/[drrrr]^{f_U} \ar@/_1pc/[ddrr]_\tau \\
  && X_U\times_{U,\sigma} U \ar[rr]\ar[d] && U\ar[d]^{\sigma}\\
                          && X_U\ar[rr]^{f_U}&& U
}$$
induces an $\mathrm{Aut}_\sk(F)$-equivariant isomorphism 
$$\Theta(\tau,\sigma): \mathbf{Iso}_U(X_U, F\times_\sk U) \to\mathbf{Iso}_U(X^{\sigma}_U, F\times_\sk U)= \mathbf{Iso}_U(X_U, F\times_\sk U)\times_{U,\sigma} U.$$
 Then the isomorphism
$$\tau^*=p_1\circ \Theta(\tau,\sigma):  \mathbf{Iso}_U(X_U, F\times_\sk U)\simeq  \mathbf{Iso}_U(X_U, F\times_\sk U)$$
({\it i.e.},  $\Theta(\tau,\sigma)=\tau^*\times \iota)$ is such that $\iota\cdot\tau^*=\sigma\cdot \iota$ . Namely, $(\tau^*,\sigma)$ is contained in $\ea(\mathbf{Iso}_U(X_U, F\times_\sk U)/U)$.
$$
\xymatrix{
 \mathbf{Iso}_U(X_U, F\times_\sk U) \ar[d]^\iota\ar[rr]^{\tau^*} &&  \mathbf{Iso}_U(X_U, F\times_\sk U)\ar[d]^\iota\\
U\ar[rr]^\sigma && U}
$$
As mentioned the above isomorphism $\tau^*$ is $\mathrm{Aut}_\sk(F)$-equivariant,  $(\tau^*,\sigma)$ is in fact contained in the subgroup $\et(T/U)$, see the following commutative diagram below. 
$$
\xymatrix{
\mathbf{Iso}_U(X_U, F\times_\sk U)\ar[rr]^{\gamma^*} \ar[d]^{\Theta(\tau,\sigma)} \ar@/_5pc/[dd]_{\tau^*\times \iota}&& \mathbf{Iso}_U(X_U, F\times_\sk U) \ar[d]^{\Theta(\tau,\sigma)}  \ar@/^5pc/[dd]^{\tau^*\times \iota}\\
\mathbf{Iso}_U(X_U^\sigma, F\times_\sk U)\ar@{=}[d]\ar[rr]^{\gamma^\sharp} && \mathbf{Iso}_U(X_U^\sigma, F\times_\sk U)\ar@{=}[d]\\
 \mathbf{Iso}_U(X_U, F\times_\sk U)\times_{U,\sigma} U \ar[rr]^{\gamma^*\times \mathrm{id}} &&  \mathbf{Iso}_U(X_U, F\times_\sk U)\times_{U,\sigma} U
}
$$
Here $\gamma^\sharp$ is the automorphism induced by the automorphism  $\gamma\in \mathrm{Aut}_\sk(F)$ on the target space $F\times_\sk U$ of $\mathbf{Iso}_U(X_U\times_{U,\sigma} U, F\times_\sk U)$ similar to $\gamma^*$. 

In particular, $\Bt(T/U)$ contains the infinite subgroup $\B(X_U/U)=\B(X/C)$ and we are done.
\end{proof}
Following Proposition~\ref{Prop: infinite Bt g>2} and Theorem~\ref{Thm: main on torsors}, there is a subgroup $\Gamma\subseteq G$ so that
the $G$-torsor $T$ furthermore reduces to a $\Gamma$-torsor $T'$ satisfying:
\begin{itemize}
\item $T'$ is irreducible and smooth over $\sk$;

\item $\et(T'/U)$ is infinite;

\item the $\Gamma$-action on $T'$ extends to its normal compactification $T'\subseteq C'$.
\end{itemize}
In particular, we have the following picture after (\ref{equ: product-quotient g>2}):
$$
\xymatrix{
(F\times_\sk C')/\Gamma \ar[d]^{p_2} & (F\times_\sk T')/\Gamma \ar[d]^{p_2} \ar@{=}[r]  \ar@{_{(}->}[l]& (F\times_\sk T)/G \ar[rr]^\simeq \ar[d]^{p_2}&& X_U  \ar[d]^{f_U}\\
C=C'/\Lambda &\ar@{_{(}->}[l] T'/\Gamma\ar@{=}[r]                         &  T/G \ar@{=}[rr] && U
}
$$
 
In summary, we have the following theorem.
\begin{Thm}
Let $f: X\to C$ be a relatively minimal surface fibration of  genus $g\ge 2$ with infinite $\B(X/C)$. Taking $U\subseteq C$ to be the smooth locus of $f$, then there are the following data:
\begin{enumerate}[i).]
\item    a smooth curve $F$ of genus $g$;

\item   an \'etale Galois cover $\iota: T'\to U$ with Galois group $\Gamma\subseteq \mathrm{Aut}_\sk(F)$;
 
\item  this Galois cover $\iota: T'\to U$, considered as a $\Gamma$-torsor, has an infinite $\et(T'/U)$.
\end{enumerate}
With these data, the fibration $f_U$ and $f$ is given as follows:
$$
\xymatrix{
X_U\ar@{=}[rr] \ar[d]^{f_U} && (F\times_\sk T')/\Gamma \ar[d]^{p_2}\\
U  \ar@{=}[rr]                         &&  T'/\Gamma
}\,\,\,\xymatrix{
X \ar@{-->}[rrrr]^{\text{minimal resolution}}    \ar[d]^{f } &&&& (F\times_\sk C')/\Gamma \ar[d]^{p_2}\\
 C  \ar@{=}[rrrr]                         &&&&  C'/\Gamma
}
$$
Here $C'$ is the normal compactification of $T'$ and $\Gamma$-acts on $F\times_\sk T'$ and $F\times_\sk C'$ both diagonally.
\end{Thm}

Note conversely, given the above data $(T',U,\Gamma, F)$, we can easily check that $\Bt(T'/U)$ is contained in  $\B(f:(F\times_\sk U')/\Gamma \to U)$. So the above data $(T',U,\Gamma, F)$ is a sufficient and necessary condition to construct fibrations $f: X\to C$ as desired. Now applying the classification of torsors with infinite equivariant automorphism group given in \S~\ref{Subsec: classification of torsors}, we have the following classification theorem. 
\begin{Thm}\label{Thm: main in g>2}
Up to isomorphism, there are exactly four kinds of relatively minimal surface fibration $f:X\to C$ of genus $g\ge 2$ with $\B(X/C)$ infinite as below.
\begin{enumerate}[(A.)]
\item There is an elliptic curve $E'$ over $\sk$, a finite subgroup $\Gamma\subseteq E'(\sk)$ and a smooth curve $F$ of genus $g$ admitting a faithful $\Gamma$-action so that $$f=p_2: X:=(F\times_\sk E')/\Gamma\to C:=E'/\Gamma.$$ Here $\Gamma$ acts on $E'$ by natural translations and acts diagonally on $F\times_\sk E'$. 

In this case $C$ is an elliptic curve, $X$ is minimal, $\kappa(X)=1$ and $f$ is smooth.

\item  The fibration is $f=p_2: X:=F\times_\sk \p \to C:=\p$ for a smooth curve $F$ of genus $g$. 

\item ($\mathrm{char}.(\sk)=p>0$)  There is a non-trivial finite dimensional $\mathbb{F}_p$-linear subspace $\Gamma\subseteq \sk$ and a smooth curve $F$ of genus $g$ admitting a faithful $\Gamma$-action so that $f: X\to C$ is as follows
$$\xymatrix{
X\ar[rrrrrr]_{\vartheta}^{\text{minimal resolution of singularities}} \ar[d]^f&&&&&& (F\times_\sk \p)/\Gamma\ar[d]^{p_2}\\
C\ar@{=}[rrrrrr] &&&&&& \p/\Gamma
}$$ Here $\Gamma$ acts on $\p$ additively (cf. Example~\ref{Exa: example of group actions}) and acts diagonally on $F\times_\sk \p$. 

In this case $C\simeq \p, \kappa(X)=-\infty$, $f$ has a unique singular fibre at $\infty$, and the Albanese map of $X$ factors through the fibration $$p_1\circ \vartheta: X\stackrel{\vartheta}{\to} (F\times_\sk \p)/\Gamma \stackrel{p_1}{\to}  F/\Gamma.$$ In particular, $q(X)=g(F/\Gamma)$. Moreover, $X$ is minimal if and only if the $\Gamma$-action on $F$ is free.

\item  There is some $n>1$ (in case $\mathrm{char}.(\sk)=p>0$, $n$ is prime to $p$), a smooth curve $F$ of genus $g$ admitting a faithful $\mu_n(\sk)$-action so that $f: X\to C$ is as follows
$$\xymatrix{
X\ar[rrrrrr]_{\vartheta}^{\text{minimal resolution of singularities}} \ar[d]^f&&&&&& (F\times_\sk \p)/\Gamma\ar[d]^{p_2}\\
C\ar@{=}[rrrrrr] &&&&&& \p/\Gamma
}$$ Here $\mu_n(\sk)$ acts on $\p$ multiplicatively (cf. Example~\ref{Exa: example of group actions}) and acts diagonally on $F\times_\sk \p$. 

In this case $C\simeq \p, \kappa(X)=-\infty$, $f$ has two singular fibres at $0,\infty$, and the Albanese map of $X$ factors through the fibration $$p_1\circ \vartheta: X\stackrel{\vartheta}{\to} (F\times_\sk \p)/\Gamma \stackrel{p_1}{\to}  F/\Gamma.$$ In particular, $q(X)=g(F/\Gamma)$. Moreover, $X$ is minimal if and only if the $\Gamma$-action on $F$ is free.   
\end{enumerate}
\end{Thm}

\section{Surface fibration with a large B-automorphism group, $g=1$}\label{Sec: B-infinite, g=1}
We then study the relatively minimal genus one fibration $f:X\to C$ with infinite $\B(X/C)$. As the automorphism group of a genus one curve is no longer rigid, the approach in the previous section does not apply.
\subsection{Statement of the main result}
\begin{Thm}\label{Thm: main+main}
Given a relatively minimal genus one fibration $f: X\to C$ with infinite $\B(X/C)$. Take $U\subseteq C$ to be the smooth locus of $f$. Then 
there is a an elliptic curve $E$ over $\sk$, a finite subgroup scheme $\Lambda \subseteq  \mathbf{Aut} _\sk(E)$, a (fppf-)$\Lambda$-torsor $\nu: T\to U$ over $U$ so that 
\begin{enumerate}[a).]

\item $\et(T/U)$ (equivalently $\Bt(T/U)$) is infinite;

\item $T$ itself is irreducible and smooth over $\sk$; 

\item let $T\subseteq C'$ be the normal compactification of $T$, then the $\Lambda$-action on $T$ extends to $C'$. 
\end{enumerate}
With these data, the fibration $f: X\to C$ is given as follows:
$$\xymatrix{
X\ar[d]^f \ar[rrrrrr]_{\text{minimal resolution of singularities}}^\vartheta &&&&&& (E\times_\sk C')/\Lambda\ar[d]^{p_2}\\ 
C\ar@{=}[rrrrrr]&&&&&& C'/\Lambda}
$$
Here $\Lambda$-acts on $E\times_\sk C'$ diagonally.
\end{Thm}
\begin{Cor}\label{Cor: last}
In the above theorem, we have $\kappa(X)=0$ if $g(C)=1$ and $\kappa(X)=-\infty$ if $g(C)=0$. 
\end{Cor}
\begin{proof}
As $\B(X/C)\subseteq \mathrm{Aut}_\sk(U)$ is infinite, $U$ is one of the four curves: an elliptic curve, $\p, \A$ or $\Aa$.

First if $U$ is an elliptic curve, then $f$ is a smooth isotrivial genus one fibration, we have $K_X=0$ and hence $\kappa(X)=0$. 

Then in the rest cases, according to the classification in Proposition~\ref{Prop: classification of torsors with infinite Bt}, we have $C'\simeq \p$ is a rational curve. Note there is another fibration $$p_1: X\stackrel{\vartheta}{\to}(E\times_\sk C')/\Lambda \to E/\Lambda$$ whose general closed fibres are isomorphic to $C'$ (cf. Corollary~\ref{Cor: quotient by gp}). The surface $X$ is a ruled surface and hence $\kappa(X)=0$.
\end{proof}

In \S~\ref{Subsec: proof of main+main}, we give a proof of Theorem~\ref{Thm: main+main} and then in \S~\ref{Subsec: classification in genus one}, we shall give a complete classification of relatively minimal genus one fibration $f: X\to C$ with infinite $\B(X/C)$ after Theorem~\ref{Thm: main+main}. Finally, we shall discuss an application to surface of Kodaira dimension one in \S~\ref{Subsec: k(S)=1}.

\subsection{Proof of Theorem~\ref{Thm: main+main}}\label{Subsec: proof of main+main}
Now let $f:X\to C$ be a relatively minimal genus one surface fibration with infinite $\B(X/C)$, $U\subseteq C$  be the smooth locus of $f$ and $\mu: J\to C$ be its Jacobian fibration. We are now going to prove Theorem~\ref{Thm: main+main}. As mentioned in \S~\ref{Sec: genus one fibration}, $X_U=\mathbf{Pic}^1_{X_U/U}$ is a torsor under its Jacobian $J_U:=\mathbf{Pic}^0_{X_U/U}$. We will denote by $\beta: J_U\times_U X_U\to X_U$ this group action.  

Similar to the genus $g\ge 2$ case, we first have the triviality of the Jacobian fibration whose proof   is the same as that of  Lemma~\ref{Lem: basic}.
\begin{Lem}\label{Lem: triviality of Jacobian}
The fibration $f: X\to C$ is isotrivial. In other words, there is an elliptic curve $E$ over $\sk$ so that all closed fibres of $f_U$ is isomorphic to $E$. In particular, the fibration $\mu: J\to C$ is also isotrivial and all of its smooth closed fibre is also isomorphic to $E$. 
\end{Lem}
The proof of Theorem~\ref{Thm: main+main} consists of two parts.
\begin{itemize}
\item  In \S~\ref{subsub: no name} we will construct a torsor $\widetilde{Z}$ over $U$ under a subgroup scheme $\Lambda_0\subseteq \mathbf{Aut}_\sk(E)$ so that $$f_U=p_2: X_U=(E\times_\sk \widetilde{Z})/\Lambda_0\to U=\widetilde{Z}/\Lambda_0.$$ Here $\Lambda_0$ acts diagonally on $E\times_\sk \widetilde{Z}$. This construction works for all isotrivial relatively minimal genus one fibration and is a characteristic-free generalization of \cite[\S~VI]{Beauville-book}. But the torsor $\widetilde{Z}$ we have here is possibly reducible and non-reduced.  We then claim a lemma (Lemma~\ref{Lemma: infinite et for wT}) saying that $\et(\widetilde{Z} /U)$ is infinite under the assumption $\B(X/C)$ is infinite. With this lemma, we prove Theorem~\ref{Thm: main+main}.

\item The \S~\ref{Subsub: proof of lemma infinite bt} is devoted to prove Lemma~\ref{Lemma: infinite et for wT}.
\end{itemize}

\subsubsection{Product-quotient realization of genus one isotrivial fibrations} \label{subsub: no name}
As $X_U$ is a torsor under $J_U$, it is represented by an element $o_{X_U}\subseteq H^1_{\acute{e}t}(U, E)=H^1_{\mathrm{fl}}(U, E)$ (cf. \S~\ref{Sec: genus one fibration} and \cite{Grothendieck-Brauer-groupIII} for the equality of \'etale and flat cohomology). Moreover, this element $o_{X_U}$ is killed by some $n\in \N_+$ (cf. Proposition~\ref{Prop: torsor of genus one is torsion}). Then the following  exact sequence in fppf topology
$$
0\to J_U[n]\to J_U\stackrel{\cdot n}{\to} J_U\to 0
$$
gives us a surjective map
$$
H^1_{\mathrm{fl}}(U,J_U[n])\to H^1_{\mathrm{fl}}(U,J_U)[n].
$$
In other words, the $J_U$-torsor $X_U$ reduces to a certain $J_U[n]$-torsor. We now construct this torsor explicitly. As $n$ kills $o_{X_U}$, the $J_U$-torsor $\mathbf{Pic}^n_{X_U/U}$ is trivial. We can thus fix an arbitrary  $s: U\to \mathbf{Pic}^n_{X_U/U}$. Then the divisor $Z\subseteq X_U$ which is the pull back of  the section $s$ under the multiplicative by $n$ map is the $J_U[n]$-torsor desired.
\begin{equation}\label{equ: definition of Z}
\xymatrix{
                Z\ar@{^{(}->}[d]           \ar[rr] && U\ar@{^{(}->}[d]^s  & J_U[n]\times_U Z\ar@{^{(}->}[d]\ar[rr]^{\beta|_Z} && Z\ar@{^{(}->}[d]\\
X_U=\mathbf{Pic}^1_{X_U/U} \ar[rr]^{\cdot n} && \mathbf{Pic}^n_{X_U/U}          & J_U\times_U X \ar[rr]^{\beta} && X
}
\end{equation}
In other words, we have the right picture above and the natural $J_U$-action on $X_U$ induces an isomorphism
\begin{equation}\label{equ: torsor isomorpism}
(J_U\times_U Z)/J_U[n]\simeq X_U,
\end{equation}
 where $J_U[n]$-acts on $J_U\times_U Z$ diagonally.
 
\begin{Rem} 
Although the above construction of the reduction $Z$ of $X_U$ is classical, we would like to explain a little more why such $Z$ is a torsor under $J_U[n]$. Functorially, by losing some rigorousness, $s$ is represented by an invertible sheaf $\cL_s$ of relative degree $n$ on $X_U$. For any $U$-scheme $W$, a section $\xi$ in $X_U(W)$ gives a section, still denoted by $\xi$, of $X_U\times_U W \stackrel{p_2}{\to} W$. This section $\xi$ is an effective relative divisor of degree one  on $X_U\times_U W$ and hence gives an invertible sheaf $\cL_\xi:=\cO_{X_U\times_UW}(\xi(W))$ on $X_U\times_U W$ of relative degree one.  By construction, the subset $Z(W)\subseteq X_U(W)$ is consisting of those $\xi$ so that $\cL_{\xi}^n\simeq \cL_s|_{X_U\times_U W}$ fibrewisely. Note that by this description, all such $\cL_\xi, \xi\in Z(W)$ is a principal set of the group $J_U[n](W)=J_U(W)[n]$ consisting of relative invertible sheaves $\cM$ on $X_U\times_U W$ so that $\cM^n\simeq \cO_{X_U\times_U W}$ fibrewisely. This explains the $J_U[n]$-torsor  structure on $T$.
\end{Rem}

\medskip

Now note the group automorphism of $E$ ($E$ is given in Lemma~\ref{Lem: triviality of Jacobian}) is also rigid. So following the isotriviality in Lemma~\ref{Lem: triviality of Jacobian}, the functor $$\underline{\mathrm{Iso}}^{\mathrm{gp}}_U(J_U,E\times_\sk U): \mathbf{Sch}/U \to \mathbf{Set},\,\,\, W\mapsto \mathrm{Iso}^{\mathrm{gp}}_W(J_U\times_U W, E\times_\sk W)$$ is represent by an \'etale Galois cover $\iota: \iso\to U$ 
whose Galois group is canonically identified with $\mathrm{Aut}_\sk^{\mathrm{gp}}(E)$. Similar to the $g\ge 2$ case, the Galois group $\mathrm{Aut}_\sk^{\mathrm{gp}}(E)$-action on $\iso$ is functorially give as follows. For any $\gamma\in \mathrm{Aut}_\sk^{\mathrm{gp}}(E)$ and group automorphism $\zeta:  J_U\times_U W\stackrel{\sim}{\to} E\times_\sk W \in \iso(W)$ the element $\gamma\cdot \zeta$ is the following composition of  isomorphisms
\begin{equation}\label{equ:  definition of gamma*}
\gamma\cdot \zeta: J_U\times_U W\stackrel{\zeta}{\to} E\times_\sk W \stackrel{\gamma\times \mathrm{id}}{\longrightarrow} E\times_\sk W  \in \underline{\mathrm{Iso}}^{\mathrm{gp}}_U(X_U, E\times_\sk U)(W).
\end{equation}
To avoid confusions, we  denote by $\gamma^*: \mathbf{Iso}^{\mathrm{gp}}_U(X_U, E\times_\sk U) \to  \mathbf{Iso}^{\mathrm{gp}}_U(X_U, E\times_\sk U)$ this Galois action  for any $\gamma\in \mathrm{Aut}_\sk^{\mathrm{gp}}(E)$. Following this Galois action, taking $$\phi: X_U\times_U\mathbf{Iso}^{\mathrm{gp}}_U(X_U, F\times_\sk U)\to E\times_\sk\mathbf{Iso}^{\mathrm{gp}}_U(X_U, E\times_\sk U) $$ the universal isomorphism, then we have the following commutative diagram for any $\gamma\in \mathrm{Aut}_\sk(F)$:
\begin{equation}\label{eq: gamma*}
\xymatrix{
\widetilde{J}\ar@{=}[r]& J_U\times_U\iso\ar[d]^{\mathrm{id}\times \gamma^*} \ar[rr]^{\phi} && E\times_\sk \iso\ar[d]^{\gamma\times \gamma^*}\\ 
\widetilde{J}\ar@{=}[r]&J_U\times_U\iso \ar[rr]^\phi&& E\times_\sk \iso.
}
\end{equation}

In the following, for simplicity, we denote by $\is:=\iso$, $\widetilde{Z}:=Z\times_U\is $ and $ \widetilde{J}:=J_U\times_U\is$. Then both  $\widetilde{Z},\widetilde{J}$ admit the following group scheme actions:
\begin{itemize}
\item the Galois group $\mathrm{Aut}_\sk^{\mathrm{gp}}(E)$-action on $\widetilde{Z}, \widetilde{J}$:
\begin{align*}
\gamma\mapsto \gamma^{\sharp}:     \widetilde{Z}=Z\times_U \is \stackrel{\mathrm{id}\times \gamma^*}{\longrightarrow}Z\times_U\is =\widetilde{T},\\
\gamma\mapsto \gamma^{\sharp}:  \widetilde{J}=J_U\times_U \is \stackrel{\mathrm{id}\times \gamma^*}{\longrightarrow}J_U\times_U\is =\widetilde{J}.
\end{align*}

\item an $E$-action $\widetilde{m}: E\times_\sk \widetilde{J}\to \widetilde{J}$ on $\widetilde{J}$ induced by the isomorphism $\phi$ as the following picture:
$$
\xymatrix{
\widetilde{J}\times_\is \widetilde{J}\ar@{=}[r]\ar[d]^{\phi\times \mathrm{id}}_\simeq & (J_U\times_U J_U)\times_U \is \ar[rrr]^{\,\,\,\,\,\,\,\, (\text{group multiplication})\times \mathrm{id}}  &&&  J_U\times_\sk \is  \ar@{=}[d] \\
(E\times_\sk \is)\times_{\is} \widetilde{J} \ar@{=}[r] & E\times_\sk  \widetilde{J} \ar[rrr]^{\widetilde{m}} &&& \widetilde{J}
} $$
\item an  $E[n]$-action $\widetilde{\beta}: E[n]\times_\sk \widetilde{Z}\to \widetilde{Z}$ on $\widetilde{Z}$  given by the following picture:
$$\xymatrix{
(J_U[n]\times_U Z)\times_U \is \ar[rr]^{\beta|_Z\times \mathrm{id}} \ar@{=}[d] &&  Z\times_U\is =\widetilde{Z} \ar@{=}[d]\\
  \widetilde{J}[n] \times_{\is} \widetilde{Z} \ar[rr] \ar[d]^{\phi\times\mathrm{id}} && \widetilde{Z} \ar@{=}[d]\\
(E[n]\times_\sk \is)\times_{\is} \widetilde{Z} \ar@{=}[r] & E[n]\times_\sk \widetilde{Z} \ar[r]^{\widetilde{\beta}} & \widetilde{Z}
}
$$
\end{itemize}
On the other hand, there is a third group action as below. Taking the subgroup scheme $$\Lambda_0:=E[n]\rtimes_\sk \mathrm{Aut}_\sk^{\mathrm{gp}}(E)\subseteq E\rtimes_\sk \mathrm{Aut}_\sk^{\mathrm{gp}}(E)= \mathbf{Aut}_\sk(E),$$ as $\Lambda_0$ is a subgroup scheme of $\mathbf{Aut}_\sk(E)$, $\Lambda_0$ can firstly act on $E$ naturally and then it can also act  on $\is$ by the canonical quotient homomorphism $\Lambda_0\to  \mathrm{Aut}_\sk^{\mathrm{gp}}(E)=\mathrm{Aut}_U(\is)$. Therefore $\Lambda_0$ can acts on $E\times_\sk\is$ diagonally. 

Using the identification $\phi$,  the two group scheme actions by $E[n]$ (induced by the $E$-action) and $\mathrm{Aut}_\sk^{\mathrm{gp}}(E)$ on $\widetilde{J}$ are in fact identified with the restriction of the diagonal $\Lambda_0=E[n]\rtimes \mathrm{Aut}_\sk^{\mathrm{gp}}(E)$-action on  $E\times_\sk \is$ to its semi-direct components $E[n]$ and $\mathrm{Aut}_\sk^{\mathrm{gp}}(E)$ by the following two commutative diagrams easily obtained. 
{\small$$
\xymatrix{
E[n]\times_\sk \widetilde{J}\ar[d]^{\mathrm{id}\times \phi}\ar[rrr]^{\widetilde{m}} &&&\widetilde{J}\ar[d]^{  \phi} & \widetilde{J}\ar[d]^{\phi} \ar[r]^{\gamma^\sharp} & \widetilde{J}\ar[d]^\phi \\
E[n]\times_\sk (E\times_\sk \is) \ar[rrr]^{\text{ group multiplication}}  &&&E\times_\sk \is & E\times_\sk \is\ar[r]^{\gamma\times \gamma^*} &   E\times_\sk \is
}$$}
In other words, the two group actions on $\widetilde{J}$ can be glued together into a single $\Lambda_0$-action.  In fact, not only the actions on $\widetilde{J}$, but also that on $\widetilde{Z}$ glue.
\begin{Lem}\label{Lem: action glue}
The above  $E[n]$ and $\mathrm{Aut}_\sk^{\mathrm{gp}}(E)$ actions on $ \widetilde{Z}$ glue  altogether into a $\Lambda_0$-action on $\widetilde{Z}$. Moreover, this time $\widetilde{Z}$ is a $\Lambda_0$-torsor over $U$. 
\end{Lem}
\begin{proof}
First we check that the two actions glue to a $\Lambda_0$-action.  Given a $\gamma\in \mathrm{Aut}_\sk^{\mathrm{gp}}(E)$, from (\ref{eq: gamma*}) and the commutative diagrams above,  we obtain the following large  commutative diagram:
$$
\xymatrix{
(J_U[n]\times_\sk \is) \times_{\is} \widetilde{Z}\ar@/_6pc/[dd]_{\phi\times \mathrm{id}}\ar[d]^{(\mathrm{id}\times \gamma^*)\times_{\gamma^*} \gamma^\sharp}\ar[rrr]^{\beta\times \mathrm{id}} &&&\widetilde{Z} \ar[d]^{\gamma^\sharp} \ar@/^2pc/[dd]^{\mathrm{id}}\\ 
(J_U[n]\times_\sk \is) \times_{\is} \widetilde{Z}      \ar@/_6pc/[dd]_{\phi\times \mathrm{id}}     \ar[rrr]^{\beta\times \mathrm{id}} &&&\widetilde{Z}\ar@/^2pc/[dd]^{\mathrm{id}}\\
 (E[n]\times_\sk \is)\times_{\is}\widetilde{Z}\ar[d]^{(\gamma,\gamma^*)\times_{\gamma^*} \gamma^\sharp} \ar@{=}[r] & E[n]\times_\sk \widetilde{Z}\ar[d]^{\gamma\times \gamma^\sharp} \ar[rr]^{\widetilde{\beta}} && \widetilde{Z} \ar[d]^{\gamma^\sharp}\\
  (E[n]\times_\sk \is)\times_{\is}\widetilde{Z} \ar@{=}[r] &E[n]\times_\sk \widetilde{Z} \ar[rr]^{\widetilde{\beta}} && \widetilde{T}
}
$$
In this commutative diagram, the box in bottom right-hand corner gives
$$
\gamma^\sharp (e_0\cdot \widetilde{z})=\gamma(e_0)\cdot (\gamma^\sharp (\widetilde{z})) \,\,\, \forall \,\,\, W\in\mathbf{Sch}/\sk, e_0\in E[n](W), \widetilde{z}\in \widetilde{Z}(W).
$$
We then finish the proof by noticing $\gamma\cdot e_0  \cdot \gamma^{-1} =\gamma(e_0)$ in $E[n]\rtimes \mathrm{Aut}_\sk^{\mathrm{gp}}(E)$.

Then we show $\widetilde{Z}$ is a torsor under $\Lambda_0$ over $U$. In fact by construction, as $Z$ is a $J_U[n]$-torsor over $U$, the action of $E[n]$ on $\widetilde{Z}$ is also free and we have $\widetilde{Z}/E[n]= (Z/J_U[n])\times_U\is=\is$. Then our conclusion follows from the fact  $\is$ is a torsor over $U$ under the group $\mathrm{Aut}_\sk^{\mathrm{gp}}(E)=\Lambda_0/E[n]$. We are done.     
\end{proof}
\begin{Prop}\label{Prop: X realized as product quotient}
We have the following commutative diagram:
$$
\xymatrix{
(E\times_\sk \widetilde{Z})/\Lambda_0\ar[rr]^\simeq \ar[d]^{p_2}&& X_U\ar[d]^{f_U}\\
\widetilde{Z}/\Lambda_0 \ar@{=}[rr] && U.
}
$$Here $\Lambda_0$-acts on $E\times_\sk \widetilde{Z}$ diagonally.
\end{Prop}
\begin{proof}
By taking base change over $\is$, we have a canonical isomorphism 
\begin{equation}
\delta: E\times_\sk \widetilde{Z}=(E\times_\sk \is)\times_\is \widetilde{Z} \stackrel{\phi\times\mathrm{id}}{\simeq} \widetilde{J}\times_\is \widetilde{Z}=(J_U\times_U Z)\times_U \is.
\end{equation}
Now since $X_U=(J_U\times_UZ)/J_U[n]$ by (\ref{equ: torsor isomorpism}), its base change $$X_U\times_U\is=((J_U\times_UZ)/J_U[n])\times_U\is$$ is the quotient of $(J_U\times_U Z)\times_U \is\simeq E\times_\sk \widetilde{Z}$ by the subgroup scheme $J_U[n]\times_U\is =\widetilde{J}[n]\simeq E[n]\times_\sk \is$ as shown in the following picture: 
$$
\xymatrix{
E\times_\sk \widetilde{Z} \ar[r]_\simeq^{\phi\times \mathrm{id}}\ar[d] &\widetilde{J}\times_{\is} \widetilde{Z}\ar@{=}[r]\ar[d] & (J_U\times_UZ)\times_U\is\ar[d]\\
(E\times_\sk \widetilde{Z})/E[n]  \ar[r]_{\simeq}^{\phi\times \mathrm{id}} &(\widetilde{J}\times_{\is} \widetilde{Z})/\widetilde{J}[n] \ar@{=}[r]&   ((J_U\times_UZ)/J_U[n])\times_U \is 
}
$$
In other words, $X\times_U\is$ is obtained by quotient $E\times_\sk\widetilde{Z}$ by the $E[n]$-action. As $J_U[n]$ acts on $J_U\times_U Z$ diagonally, one can check in the above commutative diagram, the $E[n]$-action on $E\times_\is \widetilde{Z}$ is just the  natural diagonal $E[n]$-action, namely, the restriction of the diagonal $\Lambda_0$-action on $E\times_\sk \widetilde{Z}$ to its semi-direct component $E[n]$. 

Next, to obtain $X_U$, one needs to quotient the Galois group $\mathrm{Aut}_\sk^{\mathrm{gp}}(E)$ action on $X\times_U \is$. However, this Galois group action on the quotient space $$X\times_U \is=((J_U\times_U Z)/J_U[n])\times_U\is\simeq (E\times_\sk\widetilde{Z})/E[n]$$  already lifts to the space $ E\times_\sk \widetilde{Z}\stackrel{\delta}{\simeq} (J_U\times_UZ)\times_U \is$ and we have the following commutative diagram 
$$
\xymatrix{
E\times_\sk \widetilde{Z}\ar[d]^{\gamma\times \gamma^\sharp} \ar[rr]_{\simeq}^{\phi\times \mathrm{id}} && \widetilde{J}\times_{\is} \widetilde{Z}\ar@{=}[rr]\ar[d]^{\gamma^\sharp\times \gamma^\sharp}&& (J_U\times_U Z)\times_U \is \ar[d]^{\mathrm{id}\times \gamma^*} \\
E\times_\sk \widetilde{Z} \ar[rr]_{\simeq}^{\phi\times \mathrm{id}} && \widetilde{J}\times_{\is} \widetilde{Z}\ar@{=}[rr]&&(J_U\times_U Z)\times_U \is
}
$$
Namely this Galois action is translated into the restriction of the diagonal $\Lambda_0$-action on $E\times_\sk \widetilde{Z}$ to its semi-direct summand $\mathrm{Aut}_\sk^{\mathrm{gp}}(E)$ via the isomorphism $\delta$. So in conclusion, $X$ is isomorphic to $(E\times_\sk \widetilde{Z})/\Lambda_0$. Finally, it is easy to check the fibration $f_U$ is identify with $p_2: (E\times_\sk \widetilde{Z})/\Lambda_0\to \widetilde{Z}/\Lambda_0=U$.
\end{proof}
In summary, we have 
\begin{Prop}\label{Prop: isotrivial=pre standard}
An isotrivial genus one fibration is pre-standard type.
\end{Prop} 
\begin{proof}
By Remark~\ref{Rem: main on torsors}, by shrinking $U$ to a smaller open subset $V$ if necessary, the torsor $\widetilde{Z}|_V$ reduces to torsor $T'$ over $V$ under a subgroup scheme $G'\subseteq \Lambda_0$ so that $T'$ is  irreducible. Then we have by Proposition~\ref{Prop: X realized as product quotient}:
$$
\xymatrix{
   (E\times_\sk T_1')/G' \ar[d]^{p_2} \ar@{=}[r]  & (E\times_\sk \widetilde{Z}|_V)/\Lambda_0 \ar[rr]^\simeq \ar[d]^{p_2}&& X_V  \ar[d]^{f_V}\\
  T_1'/\Lambda\ar@{=}[r]                         &  \widetilde{Z}|_V /\Lambda_0 \ar@{=}[rr] && V 
}
$$
where all group scheme actions are diagonal. We are done.
\end{proof}
 It should be noted that up to now, we have only made use of the isotriviality of $f$ rather than the infiniteness of $\B(X/C)$.
 
Finally, we mention the usage of the infinite assumption of $\B(X/C)$ is the next lemma.
\begin{Lem}\label{Lemma: infinite et for wT}
The group $\et(\widetilde{Z}/U)$ for the $\Lambda_0$-torsor $\widetilde{Z}$  is infinite.
\end{Lem}
The proof of this lemma will be given in \S~\ref{Subsub: proof of lemma infinite bt} below. Assuming this Lemma at the moment, we  complete the proof of Theorem~\ref{Thm: main+main} as follows. 
\begin{proof}[Proof of Theorem~\ref{Thm: main+main}]
By Lemma~\ref{Lemma: infinite et for wT} and Theorem~\ref{Thm: main on torsors}, the $\Lambda_0$-torsor $\widetilde{Z}$ reduces to a 
torsor $T$ of a subgroup scheme $\Lambda\subseteq \Lambda_0$ over $U$ so that 
\begin{itemize}
\item $T$ is irreducible, smooth over $\sk$;

\item $\et(T/U)$ is infinite; and 

\item the $\Lambda$-action on $T$ extends to the normal compactification $T\subseteq C'$.
\end{itemize}

Then as $\widetilde{Z}$ reduces to $T$, we have the following commutative diagram after Proposition~\ref{Prop: X realized as product quotient}, which completes the proof.
$$
\xymatrix{
(E\times_\sk C')/\Lambda \ar[d]^{p_2} & (E\times_\sk T)/\Lambda \ar[d]^{p_2} \ar@{=}[r]  \ar@{_{(}->}[l]& (E\times_\sk \widetilde{Z})/\Lambda_0 \ar[rr]^\simeq \ar[d]^{p_2}&& X_U  \ar[d]^{f_U}\\
C=C'/\Lambda &\ar@{_{(}->}[l] T/\Lambda\ar@{=}[r]                         &  \widetilde{Z} /\Lambda_0 \ar@{=}[rr] && U
}
$$
 \end{proof}

\subsubsection{Equivariant actions} \label{Subsub: proof of lemma infinite bt}
Now we go to the proof of Lemma~\ref{Lemma: infinite et for wT}. For any $(\tau,\sigma)\in \ea(X/C)$, let us denote by
\begin{align*}
 X_U^\sigma:&=X_U\times_{U,\sigma} U,  & J_U^\sigma:&=J_U\times_{U,\sigma} U, \\ 
 \mathbf{Pic}_{X_U/U}^\sigma:&=\mathbf{Pic}_{X_U/U}\times_{U,\sigma} U, &  \is^\sigma:&=\is\times_{U,\sigma} U.
\end{align*}
Then we have canonical identifications: 
\begin{align*}
J_U^{\sigma} =\mathbf{Pic}^0_{X_U^\sigma/U}\subseteq  \mathbf{Pic}_{X_U/U}^\sigma =\mathbf{Pic}_{X_U^\sigma/U} 
\,\, \text{and}\,\,\is^{\sigma} =\mathbf{Iso}_U^{\mathrm{gp}}( \mathbf{Pic}^0_{X_U^\sigma/U}, E\times_\sk U)
\end{align*}
And we denote by $\phi^\sigma: J_U^\sigma\times_U \is^\sigma \stackrel{\sim}{\to} E\times_\sk U$ the universal isomorphism on $\is^{\sigma}$. Now the isomorphism $\tau\times f: X\stackrel{\sim}{\to} X^\sigma:=X\times_{C,\sigma} C$  induces natural isomorphisms $\Theta(\tau,\sigma),\Theta_0(\tau,\sigma)$ and $\Upsilon(\tau,\sigma)$ shown in the following picture
{\small $$
\xymatrix{
J_U=\mathbf{Pic}^0_{X_U /U} \ar[d]^{\Theta_0(\tau,\sigma)}_{\simeq}\ar@{^{(}->}[r] &  \mathbf{Pic}_{X_U/U} \ar[d]_{\simeq}^{\Theta(\tau,\sigma)} & \is=\mathbf{Iso}_U^{\mathrm{gp}}(\mathbf{Pic}^0_{X_U/U}, E\times_\sk U)\ar[d]^{\Upsilon(\tau,\sigma)}_{\simeq}\\
\mathbf{Pic}^0_{X_U^\sigma/U} \ar@{=}[d]\ar@{^{(}->}[r]& \mathbf{Pic}_{X_U^\sigma/U} \ar@{=}[d] & \mathbf{Iso}^{\mathrm{gp}}_U(\mathbf{Pic}^0_{X_U^\sigma/U}, E\times_\sk U)\ar@{=}[d]\\
J_U^\sigma  \ar@{^{(}->}[r]& \mathbf{Pic}_{X_U/U}^\sigma  & \is^\sigma=\is\times_{U,\sigma} U,
}
$$} and the following commutative diagram:
\begin{equation}\label{equ: 101}
\xymatrix{
J_U\times_U\is \ar[rr]^\phi       \ar[d]^{\Theta_0(\tau,\sigma)\times \Upsilon(\tau,\sigma)}                   && E\times_\sk \is \ar[d]^{\mathrm{id}\times\Upsilon(\tau,\sigma)}\\
J_U^\sigma\times_U \is ^\sigma \ar@{=}[d]\ar[rr]^{\phi^\sigma} && E\times_\sk \is^{\sigma} \ar@{=}[d]\\
(J_U\times_U\is)\times_{U,\sigma} U \ar[rr]^{\phi\times\mathrm{id}} && (E\times_\sk \is)\times_{U,\sigma} U.
}
\end{equation}

We denote by
\begin{align*}
\tau_+&=p_1\circ \Theta(\tau,\sigma): \mathbf{Pic}_{X_U/U} \stackrel{\Theta(\tau,\sigma)}{\to}\mathbf{Pic}_{X_U^\sigma/U}=\mathbf{Pic}_{X_U/U}\times_{U,\sigma} U\stackrel{p_1}{\to} \mathbf{Pic}_{X_U/U},   \\
\tau_*&=p_1\circ \Upsilon(\tau,\sigma): \is \stackrel{\Upsilon(\tau,\sigma)}{\longrightarrow} \is^\sigma=\is\times_{U,\sigma} U\stackrel{p_1}{\to} \is.
\end{align*}
Then $(\tau_+,\sigma)$ and $(\tau_*,\sigma)$ are contained in $\ega(\mathbf{Pic}_{X_U/U}/U)$ and $\ea(\is/U)$ respectively.
More precisely, recall that $\is=\iso$ is a $\mathrm{Aut}_\sk^{\mathrm{gp}}(E)$-torsor over $U$ whose $\mathrm{Aut}_\sk^{\mathrm{gp}}(E)$-action is induced by the natural $\mathrm{Aut}_\sk^{\mathrm{gp}}(E)$-action on the source space $E\times_\sk U$. It is quite clear that $\tau^*$ preserves this group action and hence  the pair $(\tau_*,\sigma)$ is contained in the subgroup $\et(\is/U)$.

\begin{Lem}\label{Lemma: action of ea}
(1) The maps
\begin{align*}
\ea(X/C)\to \ega(\mathbf{Pic}_{X_U/U}/U)&:  (\tau,\sigma)\mapsto (\tau_+,\sigma)\\
\ea(X/C)\to \et(\is/U)&:  (\tau,\sigma)\mapsto (\tau_*,\sigma)
\end{align*}
are group homomorphisms.

(2) The map
$\ea(X/C)\times \mathbf{Pic}_{X_U/U}(U) \to \mathbf{Pic}_{X_U/U}(U)$ given by $$  (\tau,\sigma) \times s= \tau^n_+\circ s \circ \sigma^{-1}, \,\, \forall \,\, (\tau,\sigma)\in \ea(X/C), s\in \mathbf{Pic}_{X_U/U}(U)$$
   defines a group $\ea(X/C)$ action on the set  $\mathbf{Pic}_{X_U/U}(U)$ for any $n\in \Z$. 
$$
\xymatrix{
U\ar[rr]^s \ar[d]^\sigma  && \mathbf{Pic}_{X_U/U} \ar[d]^{\tau_+} \ar[rr] && U\ar[d]^\sigma\\
U\ar[rr]^{\tau_+\cdot s \cdot \sigma^{-1} } \ar[rr] && \mathbf{Pic}_{X_U/U}\ar[rr] && U\\
}
$$
\end{Lem}
\begin{proof}
The statement (1) can be checked directly. Following (1), we have 
\begin{align*}
(\tau',\sigma')\cdot ((\tau,\sigma)\cdot s)&=(\tau',\sigma')\cdot (\tau_+\circ s\circ \sigma^{-1})\\
                                          &=\tau_+'\circ \tau_+ \circ s \circ \sigma^{-1}\circ \sigma'^{-1}\\
                                          &=(\tau'\cdot\tau)_+\circ s \circ (\sigma'\cdot \sigma)^{-1}\\
                                          &=((\tau',\sigma')\cdot (\tau,\sigma))\cdot s,                                     
\end{align*}
 which gives (2).
\end{proof}
By construction, the equivariant automorphism $\tau_+$ preserves the degree and we thus denote by $\tau^n_+: \mathbf{Pic}^n_{X_U/U}\to \mathbf{Pic}^n_{X_U/U}$ the restriction of $\tau_+$ on $\mathbf{Pic}^n_{X_U/U}$. In particular, $p_1\circ\Theta_0(\tau,\sigma): J_U\to J_U$ is the same as  $\tau^0_+$ and $$\tau_+^1=\tau: \mathbf{Pic}^1_{X_U/U}=X\to X=\mathbf{Pic}^1_{X_U/U} .$$

Now we concentrate on the section $s: U\to \mathbf{Pic}^n_{X_U/U} \in \mathbf{Pic}^n_{X_U/U}(U)$ we fixed in \S~\ref{subsub: no name} (cf. the paragraph we construct $Z$), we denote by $N_s:=\mathrm{Stab}_{s}\subseteq \ea(X/C)$ the stabilizer of $s\in \mathbf{Pic}^n_{X_U/U}(U)$ under the $\ea(X/C)$-action on $\mathbf{Pic}^n_{X_U/U}(U)$ constructed in Lemma~\ref{Lemma: action of ea}
\begin{Lem}
For each $(\tau,\sigma)\in N_s$, we have $\tau(Z)=Z$. 
\end{Lem}
\begin{proof}
By the construction of $\tau_+$, we have the following commutative diagram for $(\tau,\sigma)\in N_s$:
$$
\xymatrix{
  X_U=\mathbf{Pic}^1_{X_U/U}\ar[d]^{\tau_+^1=\tau}\ar[rr]^{\cdot n} &&   \mathbf{Pic}^n_{X_U/U}\ar[d]^{\tau_+} &&\ar@{_{(}->}[ll]_s  U\ar[d]^{\sigma}\\
  X_U=\mathbf{Pic}^1_{X_U/U}\ar[rr]^{\cdot n} &&   \mathbf{Pic}^n_{X_U/U} &&\ar@{_{(}->}[ll]_s U
}
$$ As a result, the divisor $Z$ defined as the pull back of $s(U)\subseteq \mathbf{Pic}^n_{X_U/U}$ by the multiplication by $n$ map, is preserved by the action $\tau_+|_{X_U}=\tau$. 
\end{proof}
As $\tau(Z)=Z$ and noticing that $\tau_+$ preserves the group structure of $\mathbf{Pic}_{X_U/U}$, we obtain the right commutative diagram from the left one below.   
\begin{equation}\label{equ:102}
\xymatrix{
J_U\times_U X_U\ar[d]^{\tau^0_+\times_\sigma \tau} \ar[rr]^\beta&& X_U \ar[d]^{\tau}  & J_U[n]  \times_U  Z \ar[d]^{\tau^0_+\times_\sigma\tau}\ar[rr]^{\beta|_Z} && Z\ar[d]^\tau\\
J_U \times_UX_U  \ar[rr]^\beta && X_U  & J_U[n]               \times_U  Z            \ar[rr]^{\beta|_Z} && Z
}
 \end{equation}

Next we construct for each $(\tau,\sigma)\in N_s$ an isomorphism $\widetilde{\tau}$ as below: 
$$\widetilde{\tau}:=\tau\times_\sigma\tau_*: \widetilde{Z}=Z\times_U\is  \stackrel{\tau\times_\sigma\tau_*}{\longrightarrow}Z\times_U\is =\widetilde{Z}.$$

\begin{Prop}
For any $(\tau,\sigma)\in N_s$, the pair $(\widetilde{\tau},\sigma)$ is contained in $\et(\widetilde{Z}/U)$. 
\end{Prop}
\begin{proof}
The pair $(\widetilde{\tau},\sigma)$ is clearly contain in $\ea(\widetilde{Z}/U)$. It remains to that $\widetilde{\tau}$  preserves the $\Lambda_0=E[n]\rtimes_\sk \mathrm{Aut}^{\mathrm{gp}}_\sk(E)$-action on $\widetilde{Z}$ given in \S~\ref{subsub: no name} (cf. Lemma~\ref{Lem: action glue}).

First we show that the action of  second semi-direct component $\mathrm{Aut}^{\mathrm{gp}}_\sk(E)$ of $\Lambda_0$ is preserved. In fact, as mentioned that $(\tau_*,\sigma)$ is contained in $\et(\is/U)$, as automorphisms of $\is$ we have $\tau_*\circ \gamma^*=\gamma^*\circ\tau_*$ for all $\gamma\in \mathrm{Aut}_\sk^{\mathrm{gp}}(E)$. In particular, the automorphism $\widetilde{\tau}=\tau\times_\sigma \tau_*:Z\times_U\is \to Z\times_U\is $ is commutative with the automorphism $\gamma^\sharp=\mathrm{id}\times \gamma^*:Z\times_U\is \to Z\times_U\is$. We are done.

Next, we check the commutativity of $\widetilde{\tau}$ with the $E[n]$-action. Let us similarly define 
$\widetilde{\tau}_+:=\tau_+\times_\sigma \tau^*: \widetilde{J}=J_U\times_U\is \stackrel{\tau_+\times_\sigma \tau^*}{\longrightarrow}J_U\times_U=\widetilde{J}.$

 Then the following commutative diagram
\begin{equation}\label{eq 3}
\xymatrix{
\widetilde{J} \ar[d]^{\widetilde{\tau}_+} \ar[rr]^\phi && E\times_\sk \is\ar[d]^{\mathrm{id}\times \tau^*} \\
\widetilde{J}\ar[rr]^\phi && E\times_\sk\is
 }
\end{equation} 
follows immediately from the next commutative diagram equivalent to  (\ref{equ: 101}):
 \begin{equation}
\xymatrix{
J_U\times_U\is \ar[rr]^\phi \ar@/_5pc/[dd]_{\widetilde{\tau}_+\times \alpha}       \ar[d]^{\Theta_0(\tau,\sigma)\times \Upsilon(\tau,\sigma)}                   && E\times_\sk \is \ar[d]^{\mathrm{id}\times \Upsilon(\tau,\sigma)} \ar@/^5pc/[dd]^{(\mathrm{id}\times \tau_*)\times \alpha'}\\
J_U^\sigma\times_U \is ^\sigma \ar@{=}[d]\ar[rr]^{\phi^\sigma} && E\times_\sk \is^{\sigma} \ar@{=}[d]\\
(J_U\times_U\is)\times_{U,\sigma} U \ar[rr]^{\phi\times\mathrm{id}} && (E\times_\sk \is)\times_{U,\sigma} U
}.
\end{equation}
Here $\alpha: J_U\times_U\is \to U$ and $\alpha': E\times_\sk \is \to U$ are the structure maps. Combining  (\ref{equ:102}) with (\ref{eq 3}), we have the following large commutative diagram:
  {\small $$
\xymatrix{
&\widetilde{J}[n]\times_\is \widetilde{Z}\ar[ddd]^{\widetilde{\tau}_+\times_{\tau_*}\widetilde{\tau}} \ar[ddl]_{\phi\times\mathrm{id}} \ar@{=}[r] & (J_U[n]\times_U Z)\times_U \is \ar[ddd]^{(\tau_+\times_\sigma\tau)\times_\sigma \tau_*} \ar[rr]&& \widetilde{Z}\ar[ddd]^{\tau\times_\sigma \tau_*=\widetilde{\tau}}\ar@{=}[ddl] \\
\\
{\red E[n]\times_\sk \widetilde{Z}} \ar[ddd]^{\mathrm{id}\times\widetilde{\tau}}\ar[rrr]_{\widetilde{\beta}}&&&{\red \widetilde{Z}}\ar[ddd]^{\widetilde{\tau}}\\
&\widetilde{J}[n]\times_\is \widetilde{Z}\ar[ddl]^{\phi\times\mathrm{id}}\ar@{=}[r] & (J_U[n]\times_U Z)\times_U \is \ar[rr]&& \widetilde{Z}\ar@{=}[ddl] \\
\\
{\red E[n]\times_\sk \widetilde{Z}} \ar[rrr]_{\widetilde{\beta}} &&&{\red \widetilde{Z}}
}
$$ }   
The square box with red vertices says that $\widetilde{\tau}$ is commutative with the $E[n]$-action. We are done.
\end{proof}
This proposition gives an immediate corollary.
\begin{Cor}\label{Cor: Image of N_s contained in B}
The  subgroup $$\mathrm{Im}(f_\sharp|_{N_s}: N_s\to \B(X/C),\,\,\, (\tau,\sigma)\mapsto \sigma )\subseteq \mathrm{Aut}_\sk(U)$$ is contained in $\Bt(\widetilde{Z}/U)$.
\end{Cor}  
The next proposition finishes the proof of Lemma~\ref{Lemma: infinite et for wT}.
\begin{Prop}\label{Prop: finite kernel and cokernel}
The kernel and cokernel of the homomorphism $$f_\sharp|_{N_s}: N_s\to \B(X/C), \,\,\, (\tau,\sigma)\to \sigma$$ are both finite. In particular, $\Bt(\widetilde{Z}/U)$ is infinite.
\end{Prop}
\begin{proof}
First from the exact sequence (\ref{ext: intro-fundamental}):
\begin{equation*}
1\to \mathrm{Aut}_C(X)\to \ea(X/C)\stackrel{f_\sharp}{\to} \B(X/C)\to 1. 
\end{equation*}
we need only to prove that:
\begin{enumerate}[(a)]
\item the intersection of $N_s$ with $\mathrm{Aut}_C(X)$ is finite;

\item the subgroup $G:=<N_s,\mathrm{Aut}_C(X) >$ has a finite index in $\ea(X/C)$.
\end{enumerate}
Note that the subset  $\ea(X/C)\cdot s \subseteq \mathbf{Pic}^n_{X_U/U}(U)$ is canonically identified with the set $\ea(X/C)/N_s$ of left $N_s$-cosets. Then the set of $\mathrm{Aut}_C(X)$-orbits in $\ea(X/C)\cdot s$ is canonically identified with the set $\mathrm{Aut}_C(X)\backslash   \ea(X/C)/N_S$ of double cosets. As $\mathrm{Aut}_C(X)$ is normal,   $\mathrm{Aut}_C(X)\backslash   \ea(X/C)/N_S$ is the same as $\ea(X/C)/G$. So (b) is implied by 
\begin{enumerate}[(b1)]
\item there is only finitely many $\mathrm{Aut}_C(X)$-orbits in $\mathbf{Pic}^n_{X_U/U}(U)$.
\end{enumerate}
On the other hand, we have another exact sequence (\ref{ext: aut of genus one})):
\begin{align*}
1\to J_U(U) \to \mathrm{Aut}_C(X) \stackrel{\nu}{\to} \mathrm{Aut}^{\mathrm{gp}}_K(J_\eta)
\end{align*} 
As the group $\mathrm{Aut}^{\mathrm{gp}}_K(J_\eta)$ is finite, it suffices to prove that
\begin{enumerate}[(a')]
\item the intersection of $N_s$ with $J_U(U)$ is finite;

\item there are only finitely many $J_U(U)$-orbits in $\mathbf{Pic}^n_{X_U/U}(U)$.
\end{enumerate}
Let us then prove the two assertions. Note, as a subgroup of $\ea(X_U/U)$, the subgroup $J_U(U)$-action on $\mathbf{Pic}^n_{X_U/U}(U)$ (by Lemma~\ref{Lemma: action of ea}(3)) is different from the natural one we mentioned many times given by the multiplication in relative Picard scheme. By construction, the difference of the two actions are just by multiplication by $n$ as shown in the following commutative diagram:
$$
\xymatrix{J_U(U)\times \mathbf{Pic}^n_{X_U/U}(U) \ar[d]^{n\times \mathrm{id}} \ar[rr]^{\varepsilon':=\text{action as subgroup of $\ea(X/C)$}} &&\mathbf{Pic}^n_{X_U/U}(U)\ar@{=}[d] \\
J_U(U)\times \mathbf{Pic}^n_{X_U/U}(U)\ar@{=}[r]&\mathbf{Pic}^0_{X_U/U}(U)\times \mathbf{Pic}^n_{X_U/U}(U) \ar[r]^{\,\,\,\,\,\,\,\,\,\,\,\,\,\,\,  \varepsilon} & \mathbf{Pic}^n_{X_U/U}(U)}
$$
Here the action $\varepsilon$ is the action in relative Picard scheme and with this action. As mentioned before, with the action $\varepsilon$, $\mathbf{Pic}^n_{X_U/U}(U)$ is a trivial torsor under $J_U(U)$. So with the  $\varepsilon'$ action of $J_U(U)$ on $\mathbf{Pic}^n_{X_U/U}(U)$ we have 
\begin{itemize}
\item the  $J_U(U)$-action factors through the quotient group $J_U(U)/J_U(U)[n]$;

\item the quotient group $J_U(U)/J_U(U)[n]$ acts freely on  $\mathbf{Pic}^n_{X_U/U}(U)$; and

\item the set of  $J_U(U)$-orbits is (non-canonically) identified with the   quotient space  $J_U(U)/n\cdot J_U(U) $.
\end{itemize}
The statement (a') follows as the group $N_s\cap J_U(U)$ is the group $J_U(U)[n]$, which is finite.  

The statement (b') is as follows. Let $K:=K(C),\eta:=\Spec(K)$ be the generic point of $C$, and  $E_1$ be the $K/\sk$-trace of $J_\eta$.  No matter $E_1$ is trivial or an elliptic curve, the map $E_1(\sk)$ is divisible. Thus the group $J_U(U)/n\cdot J_U(U)$ is an $n$-torsion quotient group of the finitely generated group $J_U(U)/E_1(\sk)=J_\eta(K)/E_1(\sk)$ (by Lang-N\'eron (cf. \cite[Thm.~7.1]{Conrad06-Chow's-K/k-trace}). Therefore $J_U(U)/n\cdot J_U(U)$ is finite and we are done.    
\end{proof}

\subsection{Classification of genus one fibration with infinite B-automorphism}\label{Subsec: classification in genus one}
Now we give a complete classification of relatively minimal genus one fibration $f: X\to C$ with infinite $\B(X/C)$. 

According to Theorem~\ref{Thm: main+main}, we need to classify the set of the following data:
\begin{enumerate}[(a)]
\item a finite subgroup scheme $\Lambda\subseteq \mathbf{Aut}_\sk(E)$   for an elliptic curve $E$ over $\sk$ 

\item a $\Lambda$-torsor $\nu: T\to U$ with infinite $\et(T/U)$ for a smooth irreducible curve $U$ so that $T$ is irreducible and smooth over $\sk$.
\end{enumerate}
We have already classified in Proposition~\ref{Prop: classification of torsors with infinite Bt}, the second data (b) and it turns out that $\Lambda$ is Abelian (cf. Remark~\ref{Rem: G is abelian in classification}). Up to isomorphism, an Abelian subgroup scheme $\Lambda\subseteq \mathbf{Aut}_\sk(E)$ is elementary (cf. Definition~\ref{Def: elementary} and Lemma~\ref{Lem: el}). Namely, we can write $\Lambda=\Lambda_t\times_\sk \Gamma$ for a subgroup schemes $\Lambda_t\subseteq E$ and $\Gamma\subseteq \mathrm{Aut}_\sk^{\mathrm{gp}}(E)$ so that $\Lambda_t$ is fixed by  $\Gamma$. So we really aims to classify the following data
\begin{itemize}
\item an elliptic curve  $E$ over $\sk$ with an Abelian subgroup $\Gamma\subseteq \mathrm{Aut}_\sk^{\mathrm{gp}}(E)$ and a finite subgroup scheme $\Lambda_t\subseteq E^\Gamma$;

\item a $\Lambda:=\Lambda_t\times_\sk \Gamma$-torsor $T$ over $U$ with infinite $\Bt(T/U)$ so that $T$ is irreducible and smooth over $\sk$.
\end{itemize}
With the classification of $(E,\Gamma, E^\Gamma)$ given in \S~\ref{Sec: Appendix}, we shall now work out the complete classifications. Recall there are four possibilities for $U$: an elliptic curve, $\p, \A$ and $\Aa$.

\subsubsection{$U=C$ is an elliptic curve}\label{7.3.1}
When $U=C$ is an elliptic curve, the torsor $T=C'$ is another elliptic curve, the natural map $\nu: C'\to C$ is a homomorphism and hence $\Lambda\subseteq \mathbf{Aut}_\sk(E)$ is at the same time canonically identified with the kernel of $C'\to C$. In other words, we need to give a complete list of pairs $(E, \Lambda=\Lambda_t\times_\sk \Gamma\subseteq \mathbf{Aut}_\sk(E), C')$ along with an embedding of group schemes $\Lambda\to C'$. With these data, the fibration is  $$f=p_2: X=(E\times_\sk C')/\Lambda\to C:=C'/\Lambda.$$  
According to whether $\Gamma$ is trivial or not, we have two possibilities as below by  the Enriques-Kodaira classification (cf. \cite[\S~10]{Badescu01-algebriac-surfaces}).
\begin{enumerate}
\item[($\Gamma=\{\mathrm{id}\}$)]   $X$ is an Abelian surface (case 4 of the list in   \cite[\S~10]{Badescu01-algebriac-surfaces}). In this case $f:X\to C$ is a surjective homomorphism with smooth connected kernel.  

\item[($\Gamma\neq \{\mathrm{id}\}$)]  $X$ is a hyperelliptic  (also known as bielliptic in some literatures) surface (case 5 of the list in   \cite[\S~10]{Badescu01-algebriac-surfaces}). In this case, $f$ is the Albanese map and the numerical invariants of $X$ are either $p_g(X)=0, q(X)=1, b_1(X)=2$ or $p_g(X)=1,q(X)=2, b_1(X)=2, \Delta=2$. 
\end{enumerate}
%The second case with $\Delta>0$ is less common and we call such $X$ \emph{special}. A special surface $X$ has trivial canonical class but it is neither an Abelian surface nor a K3 surface. 
  
For the hyperelliptic surface case,  when $\mathrm{char}.(\sk)\neq 2,3$, there already exists a complete classification known as the Bagnera-de Franchis list (cf.  \cite[pp~160, \S~10]{Badescu01-algebriac-surfaces}). Moreover, the equivariant automorphism $\ea(X/C)$, which is the same as $\mathrm{Aut}_\sk(X)$ at this time has also been worked out in \cite{Bennett-Miranda90}. In characteristic $2,3$, the full list of hyperelliptic surfaces has also been given by Bombieri and Mumford in \cite{Bombieri-MumfordII77}.

\begin{Rem}
(1). In characteristic $p>0$, an ordinary elliptic curve does not contain $\alpha_p$ as a subgroup scheme. In particular, when $p=3, j(E)=0, |\Gamma|=3$ the subgroup scheme $E^\Gamma\times \Gamma \simeq \alpha_p\times \Gamma$ can not embed to any elliptic curve $C'$. The case $p=2,j(E)=1728,|\Gamma|=4$ is similar.

(2). The result surface $X$ is special if and only $\Gamma$ fixes the the space $H^0(E,\Omega_{E/\sk}^1)$ of $1$-forms. This happens if and only if $|\Gamma|$ is a power of $p$ (cf. Remark~\ref{Rem: fix 1-form}).

(3). The $\Gamma$ is always cyclic.
\end{Rem}

\subsubsection{$U=C=\p$}
According to Proposition~\ref{Prop: classification of torsors with infinite Bt}, this time the fibration $f$ is 
$$f=p_2: X=E\times_\sk\p\to \p$$ for an elliptic curve $E$.

\subsubsection{$U=\A$}
In this case $\Lambda=\Lambda_t\times \Gamma \subseteq E\rtimes_\sk \mathrm{Aut}_\sk(E)$ is a subgroup of $\mathbb{G}_a$  by  
Proposition~\ref{Prop: classification of torsors with infinite Bt}.  Note that finite subgroup scheme of $\mathbb{G}_a$ has the formation $\alpha_{p^n}\times V$ for a $\mathbb{F}_p$-linear subspace $V\subseteq \sk$. In particular, $\Gamma$ is a cyclic group of order $p$ or trivial. According to Proposition~\ref{Prop: Abelian Automorphism of elliptic curves}, we have the following classifications.

\begin{enumerate}
\item When $\Gamma=e$. Either $E$ is ordinary, $\Lambda_t=E[p](\sk)$ or $E$ is supersingular, $\Lambda_t=\mathrm{Ker}(F_{E/\sk}: E\to E^{(p)})\simeq \alpha_p$. Note that when $E$ is supersingular, $E[p]$ is not isomorphic to $\alpha_{p^2}$ and can never embeds to $\mathbb{G}_a$.

\item When $\Lambda_t=0$. Then $\Gamma$ is a cyclic group of order $p$. There are following classification:
\begin{itemize}
\item either $p=3, E: y^2=x^3-x \,\,(j(E)=0), \Gamma=<\gamma>$ with $\gamma: (x,y)\mapsto (x+1,y)$.

\item or $p=2,  E: y^2+y=x^3\,\,(j(E)=0), \Gamma=<\gamma>$ with $\gamma: (x,y)\mapsto (x,y+1)$.
\end{itemize}
 
\item Neither $\Lambda_t$ nor $\Gamma$ is trivial. There are following classification: \begin{itemize}
\item either $p=3, E: y^2=x^3-x, \Gamma=<\gamma>$ with $\gamma: (x,y)\mapsto (x+1,y)$ and $ \Lambda_t=\mathrm{Ker}(F_{E/\sk}: E\to E^{(p)})\simeq \alpha_p$.

\item or $p=2,  E: y^2+y=x^3, \Gamma=<\gamma>$ with $\gamma: (x,y)\mapsto (x,y+1)$  and $ \Lambda_t=\mathrm{Ker}(F_{E/\sk}: E\to E^{(p)})\simeq \alpha_p$.
\end{itemize}
\end{enumerate}

\subsubsection{$U=\Aa$}
In this case $\Lambda=\Lambda_t\times \Gamma \subseteq E\rtimes_\sk \mathrm{Aut}_\sk(E)$ is a subgroup of $\mathbb{G}_m$  by  
Proposition~\ref{Prop: classification of torsors with infinite Bt}.  Note that finite subgroup scheme of $\mathbb{G}_m$ has the formation $\mu_{p^n}\times \Z/m\Z$ for some $n, m $ with $(m,p)=1$. According to Proposition~\ref{Prop: Abelian Automorphism of elliptic curves}, we have the following classifications.
\begin{enumerate}
\item When $\Gamma=e$, then either $\Lambda_t\subseteq E(\sk)$ is a cyclic group of order $n$ with $(n,p)=1$, or $E$ is not supersingular,   $\Lambda_t$ is the production of  a cyclic subgroup of $E(\sk)$ of order $n$  with $(n,p)=1$ and $\mathrm{Ker}(F^k_{E/\sk}: E\to E^{(p^k)})\simeq \mu_{p^k}$;

\item When $\Lambda_t=0$, then  $\Gamma$ can be any Abelian subgroup $\mathrm{Aut}^{\mathrm{gp}}_\sk(E)$ of order prime to $p$. These are the following:
\begin{itemize}
\item $p\neq 2$, $\Gamma=\{\pm \mathrm{id}\}$;  

\item $p\neq 3, E: y^2-y=x^3-1\,\, (j(E)=0), \Gamma=<\gamma>$ with $\gamma: (x,y)\mapsto (\zeta_3x, y)$;

\item $ p\neq 2, E: y^2=x^3-x \,\,(j(E)=1728), \Gamma=<\gamma>$ with $\gamma: (x,y)\mapsto (-x, \zeta_4 y)$;

\item $p\neq 2,3, E:y^2=x^3-1\,\, (j(E)=0), \Gamma=<\gamma>$ with $\gamma: (x,y)\mapsto (\zeta_3x, -y)$.
\end{itemize} Here $\zeta_n$ is a primitive unit root of order $n$.

\item If neither $\Gamma$ nor $\Lambda_t$ are trivial, then $\Lambda_t$ must contain an additive subgroup scheme or a subgroup scheme of the same order with $\Gamma$ by Proposition~\ref{Prop: Abelian Automorphism of elliptic curves}. Thus $\Lambda$ can not embed to $\mathbb{G}_m$.
\end{enumerate}

\subsubsection{The classification theorem}
 In summary, we have the following theorem.
\begin{Thm}\label{Thm: classification of genus one}There are exactly   four kinds of relatively minimal genus one fibration $f:X\to C$ with infinite $\B(X/C)$ as follows:
\begin{enumerate}[(A).]
\item[(A-i)] The surface $X$ is an Abelian surface, $C$ is an elliptic curve and $f:X\to C$ is a surjective homomorphism with smooth connected kernel.

\item[(A-ii)] The surface $X$ is a hyperelliptic surface given in the Bagnera-DeFranchis list (cf. \S~\ref{7.3.1}).

\item[(B)] The fibration $f$ is $$f=p_2: X=:E\times_\sk \p \to \p$$ for an elliptic curve $E$ over $\sk$.

\item[(C-i)] The characteristic of $\sk$ is $p>0$, there is an elliptic curve $E$ over $\sk$, a non-trivial subgroup scheme $\Lambda\subseteq E$ along with an embedding $\Lambda\subseteq \mathbb{G}_a$, so $\Lambda$ acts on $\p$ additively. The fibration $f$ is as below:
$$f=p_2: X:=(E\times_\sk\p)/\Lambda \to C=\p/\Lambda.$$
Here $\Gamma$ acts on $E\times_\sk \p$ diagonally. The complete classification of $(E,\Lambda)$ is as follows:
\begin{itemize}
\item for ordinary $E$ over $\sk$, we take $\Lambda= E[p](\sk)$;

\item  for supersingular $E$, we take  $\mathrm{Ker}(F_{E/\sk}: E\to E^{(p)})\simeq \alpha_p$.
\end{itemize}
In these two cases, $f$ has a singular fibres at $\infty\subseteq \p$ of type $pI_0$.

\item[(C-ii)] 
In the following two pairs $(E, \Gamma\subseteq \mathrm{Aut}_\sk^{\mathrm{gp}}(E))$ of elliptic curve with order $p=\mathrm{char}.(\sk)$ cyclic automorphism group:
\begin{itemize}
\item $p=3, E: y^2=x^3-x $ , $\Gamma=<\gamma>, \gamma: (x,y)\mapsto (x+1,y)$,
 
\item $p=2, E: y^2+y=x^3 $, $\Gamma=<\gamma>, \gamma: (x,y)\mapsto (x,y+1)$,  
\end{itemize}
we take $\Lambda_t=0$ or $\mathrm{Ker}(F_{E/\sk}: E\to E^{(p)})\simeq \alpha_p$ and $\Lambda=\Lambda_t\times_\sk \Gamma$.  Fixing an embedding $\Lambda\hookrightarrow \mathbb{G}_a$ (up to automorphism of $\mathbb{G}_a$ such an embedding is unique), then $\Lambda$  acts on $\p$ additively (cf. Example~\ref{Exa: example of group actions}). The fibration $f$ is as below:
$$\xymatrix{
X\ar[rrrrrr]_{\vartheta}^{\text{minimal resolution of singularities}} \ar[d]^f&&&&&& (E\times_\sk \p)/\Lambda\ar[d]^{p_2}\\
C\ar@{=}[rrrrrr] &&&&&& \p/\Lambda
}$$ Here $\Lambda$ acts on $E\times_\sk \p$ diagonally. 

In this case, there is a unique singular fibre of both $f$ and its Jacobian fibration $\mu: J\to C$ at $\infty$. When $\Lambda_t=0$, $f$ is the same as its Jacobian and the singular fibre is of type $II^*$. If $\Lambda_t\neq 0$, the the singular fibre of $f$ is type $pII^*$ (note in characteristic zero, the only multiple fibre is type $nI_k$).

\item[(D-i)] There is an elliptic curve $E$, a subgroup scheme $\Lambda\subseteq E$ along with an embedding $\Lambda\subseteq \mathbb{G}_m$, so $\Lambda$ acts on $\p$ multiplicatively. The fibration $f$ is as below:
$$f=p_2: X:=(E\times_\sk\p)/\Lambda \to C=\p/\Lambda.$$ Here $\Gamma$ acts on $E\times_\sk \p$ diagonally. 
The complete classification of $(E,\Lambda)$ is as follows:
\begin{itemize}
\item for any $E$, we can take $\Lambda\subseteq E(\sk)$ to be a non-trivial cyclic subgroup of order   prime to $\mathrm{char}.(\sk)$;

\item when $E$ is supersingular, we can take $\Lambda$ to be either $\mathrm{Ker}(F_{E^k/\sk}: E\to E^{(p^k)})\simeq \mu_p$ or its production with a non-trivial cyclic subgroup of $E(\sk)$ of order   prime to $\mathrm{char}.(\sk)$. 
\end{itemize}
In these two cases, $f$ has two singular fibres at $0,\infty\subseteq \p$ of type $nI_0$ with $n=\dim_\sk \cO_\Lambda$. 

\item[(D-ii)] There is an elliptic curve $E$, a non-trivial cyclic group $\Gamma\subseteq \mathrm{Aut}^{\mathrm{gp}}_\sk(E)$ of order $n$ prime to $\mathrm{char}.(\sk)$. Fixing an arbitrary isomorphism $\Gamma\simeq \mu_n(\sk)$, then $\Gamma$ also acts on $\p$ multiplicatively. The fibration $f$ is as below:
$$\xymatrix{
X\ar[rrrrrr]_{\vartheta}^{\text{minimal resolution of singularities}} \ar[d]^f&&&&&& (F\times_\sk \p)/\Gamma\ar[d]^{p_2}\\
C\ar@{=}[rrrrrr] &&&&&& \p/\Gamma
}$$ Here $\Gamma$ acts on $E\times_\sk \p$ diagonally. 
The complete classification of $(E,\Gamma)$ is as follows:

\begin{itemize}
\item $p\neq 2$, $\Gamma=\{\pm 1\}$.  The singular fibres at $0,\infty$ are of type $I_0^*$;

\item $p\neq 3, E: y^2-y=x^3-1\,\, (j(E)=0), \Gamma=<\gamma>$ with $\gamma: (x,y)\mapsto (\zeta_3x, y)$. The singular fibres at $0,\infty$ are of  type  $IV$ and $IV^*$ (or vice versa) respectively;

\item $ p\neq 2, E: y^2=x^3-x \,\,(j(E)=1728), \Gamma=<\gamma>$ with $\gamma: (x,y)\mapsto (-x, \zeta_4 y)$. The singular fibres at $0,\infty$ are of type $III$ and $III^*$ (or vice versa) respectively.
\end{itemize}
 
\end{enumerate}
\end{Thm}
\subsection{On surface with $\kappa(S)=1$}\label{Subsec: k(S)=1}
Let $S$ be a minimal surface over $\sk$ of Kodaira dimension one. Then it is well known that the Iitaka fibration of $S$ is a fibration $f: S\to C$   whose general fibre (not necessarily smooth) has arithmetic genus one. When $p=2,3$ this fibration can be quasi-elliptic and otherwise it is a genus one fibration in the sense of this paper (with smooth general fibre). As this fibration is canonical, it is $\mathrm{Aut}_\sk(S)$-equivariant and thus $\ea(S/C)=\mathrm{Aut}_\sk(C)$.     A direct consequence of Corollary~\ref{Cor: last} gives that
\begin{Prop}\label{Prop: kappa=1}
If $f$ is a genus one fibration in the sense of this paper, then $\B(S/C)$ is finite. Namely, the image of $\mathrm{Aut}(S)\to \mathrm{Aut}(C)$ is finite.
\end{Prop}
\begin{Cor}\label{Cor: Jordan}
If $f$ is a genus one fibration in the sense of this paper, the $\mathrm{Aut}_\sk(S)$ has Jordan property. Namely, there is a uniform bound $N(S)\in \N_+$ such that any finite subgroup of  $\mathrm{Aut}_\sk(S)$ contains an Abelian subgroup of index at most $N(S)$.
\end{Cor}
\begin{proof}
It suffices to prove that  $\mathrm{Aut}_\sk(S)$ contains an Abelian subgroup of finite index. As $\B(S/C)$ is finite, $\mathrm{Aut}_\sk(S)$ contains an Abelian subgroup $J(C)$ of finite index by (\ref{ext: intro-fundamental}) and (\ref{ext: aut of genus one}). Here $\mu: J\to C$ is the Jacobian fibration.  We are done.
\end{proof}
In characteristic zero, Proposition~\ref{Prop: kappa=1} is proven by Prokhorov and Shramov (cf. \cite[Lem.~3.3]{P-S19}) based on the canonical bundle formula of Kodaira. The key ingredient is to rule out the possibility where the base $C=\p$, $f$ is isotrivial with at most two multiple fibres. Their argument actually implies the following  result.
\begin{Prop}\label{Prop: last}
Let $\sk$ be an algebraically closed field of characteristic zero, $f: X\to \p$ be a relatively minimal isotrivial genus one fibration with  at most two singular fibres, then $\kappa(X)<1$.
\end{Prop} 
In positive characteristics, Proposition~\ref{Prop: last} fails (cf. Example~\ref{Exa: concrete}). In other words, it is unlikely that one can prove Proposition~\ref{Prop: kappa=1} with canonical bundle formula alone. A thorough study of the structure of $f:X\to C$ with infinite $\B(X/C)$ as this paper looks necessary.

\section{Appendix: Abelian automorphism of elliptic curves}\label{Sec: Appendix}
We shall give a complete classification of pairs $(E,\Lambda)$, where $E$ is an elliptic curve over $\sk$   and $\Lambda\subseteq \mathbf{Aut}_\sk (E)$ is an Abelian subgroup scheme. 

\subsection{Abelian group-autormorphism of elliptic curves}
\subsubsection{Structure of Abelian group-automorphism group of elliptic curves} Recall 
\begin{Thm} [\protect{\cite[\S~A, Prop.~1.2]{Silverman-The-arithmetic-of-elliptic-curves}}]\label{Automorphism of elliptic curves}
Let $E$ be an elliptic curve over an algebraically closed field $\sk$ of characteristic $p$, then 
$$
|\mathrm{Aut}^{\mathrm{gp}}_{\sk}(E)|=\left\{\begin{array}{cc}
2, & j(E)\neq 0,1728 \\
4, & j(E)=1728, p\neq 2,3 \\
6, & j(E)=0, p\neq 2,3 \\
12,& j(E)=0, p=3\\
24,& j(E)=0,p=2.
\end{array}  \right.
$$ Moreover,  the group structure of $\mathrm{Aut}^{\mathrm{gp}}_{\sk}(E)$ is also clear: 
\begin{itemize}
\item when $|\mathrm{Aut}^{\mathrm{gp}}_{\sk}(E)|\le 6$, it is cyclic; and
\item when $j(E)=0,p=2,3$ the structure of $\mathrm{Aut}^{\mathrm{gp}}_{\sk}(E)$ is described in \cite[\S~A, Exer.~A.1.]{Silverman-The-arithmetic-of-elliptic-curves}.
\end{itemize}
\end{Thm}
An easy group theoretic calculation based on the explicit structure of $\mathrm{Aut}^{\mathrm{gp}}_{\sk}(E)$ given in \cite[\S~A, Exer.~A.1.]{Silverman-The-arithmetic-of-elliptic-curves}, we have the following.
\begin{Prop}\label{Prop: Abelian to cyclic}
Any Abelian subgroup of $\mathrm{Aut}^{\mathrm{gp}}_{\sk}(E)$ is a cyclic group. Two Abelian subgroups of $\mathrm{Aut}^{\mathrm{gp}}_{\sk}(E)$ of the same order are conjugate to each other. 
\end{Prop}

In other words, up to automorphism of $E$, Abelian subgroups of $\mathrm{Aut}^{\mathrm{gp}}_{\sk}(E)$ are uniquely determined by their orders.
$$
\xymatrix{
E\ar[d]_{\simeq}^\sigma && \Lambda\ar@{^{(}->}[rr] \ar[d]&& \mathrm{Aut}_\sk(E) \ar[d]^{\sigma^*=\text{conjugate by $\sigma$}} \\
E              &&\sigma \Lambda\sigma^{-1}\ar@{^{(}->}[rr] && \mathrm{Aut}_\sk(E)
}
$$

\subsubsection{List of maximal Abelian group-automorphism subgroups of elliptic curves}\label{Subsub: list} The following is the complete list of pair $(E,\Gamma)$ where $E$ is an elliptic curve with $j(E)=0$ or $1728$ and $\Gamma=<\gamma> (\text{or} <\gamma'>)\subseteq \mathrm{Aut}^{\mathrm{gp}}_{\sk}(E)$ is a maximal Abelian automorphism subgroup.
\begin{align*} 
&(j(E)=1728, p\neq 2,3) & E: y^2=x^3-x, & \,\,\gamma: (x,y)\mapsto (-x, \zeta_4 y),  |\Gamma|=4; \\
&(j(E)=0, p\neq 2,3) & E: y^2=x^3-1, &\,\, \gamma: (x,y)\mapsto (\zeta_3x, -y),  |\Gamma|=6; \\
&(j(E)=0, p=3) & E: y^2=x^3-x, & \,\,\gamma: (x,y)\mapsto (-x, \zeta_4 y),  |\Gamma|=4; \\
&& & \,\,\gamma': (x,y)\mapsto (x+1, -y),  |\Gamma'|=6; \\
&(j(E)=0, p=2) & E: y^2+y=x^3, &\,\, \gamma': (x,y)\mapsto (\zeta_3x, y+1),  |\Gamma|=6; \\
&&  \gamma: (x,y)&\mapsto (x+1, y+x+\zeta_3),  |\Gamma|=4. 
\end{align*}Here $\zeta_n$ is a primitive unit root of order $n$.
\begin{Rem}\label{Rem: fix 1-form}
In this list, one can check that $\gamma$ or $\gamma'$ fixes $H^0(E,\Omega^1_{E/\sk})$ if and only if $\mathrm{ord}(\gamma)$ or $\mathrm{ord}(\gamma')$ is a power of $p$.
\end{Rem}
\subsection{Structure of Abelian automorphism group scheme of elliptic curves}
For an elliptic curve $E$ over $\sk$, we have the semi-direct decomposition $$\mathbf{Aut}_\sk(E)=E \rtimes \mathrm{Aut}_\sk^{\mathrm{gp}}(E).$$ Therefore if given a finite subgroup scheme $\Lambda_t\subseteq E$ and a finite group $\Gamma\subseteq  \mathrm{Aut}_\sk^{\mathrm{gp}}(E)$ so that $\Lambda_t$ is preserved by the $\Gamma$-action, then $$\Lambda:=\Lambda_t\rtimes \Gamma\subseteq E \rtimes \mathrm{Aut}_\sk^{\mathrm{gp}}(E)$$ is a finite subgroup scheme of $\mathbf{Aut}_\sk(E)$.  
\begin{Def}\label{Def: elementary}
A subgroup scheme of the above formation $\Lambda=\Lambda_t\rtimes \Gamma$ is called an elementary subgroup scheme of $\mathbf{Aut}_\sk(E)$.
\end{Def}
Clearly, not all subgroup schemes are elementary. 
\begin{Lem}\label{Lem: el}
Up to translation isomorphism of $E$, any finite Abelian subgroup scheme of  $\mathbf{Aut}_\sk(E)$ is elementary. That is, let $\Lambda\subseteq \mathbf{Aut}_\sk(E)$ be a finite Abelian subgroup scheme. Then there is a suitable $a\in E(\sk)$ so that the conjugate group scheme $\Lambda^{T_a}:=T_a\cdot \Lambda\cdot T_{-a}$ is elementary.
\end{Lem}
In other words, up to the isomorphism $T_a: E\to E$, we can assume $$\Lambda=\Lambda_t\rtimes \Gamma\subseteq E \rtimes \mathrm{Aut}_\sk^{\mathrm{gp}}(E)=\mathbf{Aut}_\sk(E)$$ for a finite subgroup scheme $\Lambda_t\subseteq E$ and a finite group $\Gamma\subseteq  \mathrm{Aut}_\sk^{\mathrm{gp}}(E)$ so that $\Lambda_t$ is fixed (rather than  preserved since $\Lambda$ is Abelian) by the $\Gamma$-action.

\begin{proof}
A subgroup scheme $H\subseteq \mathbf{Aut}_\sk(E)$ is elementary if the quotient homomorphism 
$$\pi|_{H(\sk)}: H(\sk) \to \mathrm{Aut}_\sk^{\mathrm{gp}}(E) $$ admits a section over its image subgroup $\pi(H(\sk))$. 
For our $\Lambda$, the image $\pi(\Lambda(\sk))$ is Abelian and hence cyclic by Proposition~\ref{Prop: Abelian to cyclic}. If the image group $\pi(\Lambda(\sk))$ is trivial, then we are done automatically. Then we let $\phi: E\to E\in \Lambda(\sk)$ be an element so that $\overline{\phi}:=\pi(\phi)\neq \mathrm{id}$ is a generator of $\pi(\Lambda(\sk))$. So we can write $\phi=T_x\circ \overline{\phi}$ for a suitable $x\in E(\sk)$. Now as $\overline{\phi}\neq \mathrm{id}$, the map $$E(\sk)\to E(\sk): y\mapsto \overline{\phi}(y)-y$$ is surjective. In particular, there is $a\in E(\sk)$ so that $\overline{\phi}(a)-a=x$. Namely, $$T_a\circ \phi \circ T_{-a}(y)=\phi(y-a)+a=\overline{\phi}(y)-\overline{\phi}(a)+a+x=\overline{\phi}(y).$$  In other words, $\overline{\phi}$ is contained in $\Lambda^{T_a}$. We are done as $\overline{\phi}$ generates $\pi(\Lambda(\sk))$.
\end{proof}
\begin{Rem}
This proof is a copy and slight modification of that in \cite[pp.~159, \S~10]{Badescu01-algebriac-surfaces} where $p\neq 2,3$. The key is that $\pi(\Lambda(\sk))$ is cyclic.  
\end{Rem}

\subsection{Classification}
By Lemma~\ref{Lem: el}, up to translations, Abelian subgroup schemes of an elliptic are elementary. It now suffices to classify elementary Abelian subgroup scheme. Namely the triple pair $(E,\Lambda_t\subseteq E,\Gamma\subseteq \mathrm{Aut}^{\mathrm{gp}}_\sk(E))$ so that $\Lambda_t$ is fixed by $\Gamma$. To obtain such a classification, we only have to classify the pair $(E,\Gamma)$ and at the same time work out the subgroup scheme $E^\Gamma$ fixed by $\Gamma$. 

From the  list given in \S~\ref{Subsub: list}, we have the following proposition by a simple calculation.
\begin{Prop}\label{Prop: Abelian Automorphism of elliptic curves}
We have the following classification $(E,\Gamma, E^\Gamma)$ (recall by Proposition~\ref{Prop: Abelian to cyclic}, $\Gamma$ is uniquely determined by $|\Gamma|$ upto automorphism of $E$):
\begin{itemize}
\item if $j(E)\neq 0,1728$, then $\Gamma$ is either trivial or $\Gamma=\pm \mathrm{id}$.  In the latter case,  $$E^\Gamma=E[2]\simeq \left\{\begin{array}{cc}
(\Z/2\Z)^{2},  & p\neq 2;\\
\mu_2\times_\sk \Z/2\Z, & p=2.
\end{array} \right. $$ 

\item if $j(E)=1728$, $p\neq 2,3$, then $\Gamma$ is a cyclic group of order $1,2$ or $4$. In this case
$$E^\Gamma=\left\{ \begin{array}{cc}
E[2]\simeq (\Z/2\Z)^2, & |\Gamma|=2;\\
\Z/2\Z, & |\Gamma|=4. 
\end{array}\right.$$

\item if $j(E)=0, p\neq 2,3$, then $\Gamma$ is a cyclic group of order $1,2,3$ or $6$.
In this case
$$E^\Gamma=\left\{ \begin{array}{cc}
E[2]\simeq (\Z/2\Z)^2, & |\Gamma|=2;\\
\Z/3\Z  & |\Gamma|=3;\\
e  & |\Gamma|=6;\\
\end{array}\right.$$

\item if $j(E)=0,p=3$,  then $\Gamma$ is a cyclic group of order $1,2,3,4$ or $6$; 
In this case
$$E^\Gamma=\left\{ \begin{array}{cc}
E[2]\simeq (\Z/2\Z)^2, & |\Gamma|=2;\\
\mathrm{Ker}(F_{E/\sk}: E\to E^{(p)})\simeq \alpha_3 & |\Gamma|=3.\\
\Z/2\Z, & |\Gamma|=4; \\
e, & |\Gamma|=6.
\end{array}\right.$$

\item if $j(E)=0,p=2$,  then $\Gamma$ is a cyclic group of order $1,2,3,4$ or $6$. 
In this case
$$E^\Gamma=\left\{ \begin{array}{cc}
E[2] (\text{infinitesimal}), & |\Gamma|=2;\\
\Z/3\Z & |\Gamma|=3;\\
 \mathrm{Ker}(F_{E/\sk}: E\to E^{(p)})\simeq \alpha_2  & |\Gamma|=4;\\
e & |\Gamma|=6.
\end{array}\right.$$
\end{itemize}
\end{Prop}

\medskip

\begin{center}
\textbf{Acknowledgement}
\end{center}

The author would like to thank professor Jinxing Cai, Yifei Chen, Lifan Guan, Xiaoyu Su and Lei Zhang (USTC) and Lei Zhang (CUHK) for helpful communications. The author would also like to thank an anonymous referee for pointing out some mistakes in the first version of this paper.

\bibliography{ref}

\bigskip

\end{document}